\documentclass[11pt]{article}
\usepackage{graphicx}
\usepackage{amsmath}
\usepackage{amsfonts}
\usepackage{dsfont}

\newcommand{\GQ}{{\mathds{Q}}}
\newcommand{\GV}{{\mathds{V}}}
\newcommand{\GN}{{\mathds{N}}}
\newcommand{\GR}{{\mathds{R}}}
\newcommand{\GC}{{\mathds{C}}}
\newcommand{\GZ}{{\mathds{Z}}}
\newcommand{\GD}{{\mathds{D}}}
\newcommand{\GI}{{\mathds{I}}}

\newcommand{\bP}{{\mathbb{P}}}
\newcommand{\bE}{{\mathbb{E}}}
\newcommand{\bV}{{\mathbb{V}}}
\newcommand{\ve}{{\varepsilon}}

\newcommand{\cD}{{\mathcal D}}

\newcommand{\cP}{{\mathcal P}}

\newcommand{\cK}{{\mathcal K}}

\newcommand{\conver}{\mathop{\longrightarrow}}
\newcommand{\indist}{\ \conver_{\cD}\ }
\newcommand{\inprob}{\ \conver_{\cP}\ }

\newcommand{\infdd}{\ \conver_{f.d.d.}\ }
\newcommand{\ines}{\ \conver_{S}\ }
\newcommand{\indees}{\ \conver_{{\mathcal{D}(S)}\ }}
\newcommand{\inemone}{\ \conver_{M_1}\ }
\newcommand{\inessh}{\conver_{S}}
\newcommand\I{{ 1\hspace{-1,1mm}{\mathrm I}}}

\newtheorem{thm}{Theorem}[section]

\newtheorem{prop}[thm]{Proposition}
\newtheorem{lem}[thm]{Lemma}
\newtheorem{cor}[thm]{Corollary}
\newtheorem{rem}[thm]{Remark}
\newtheorem{exa}[thm]{Example}
\newtheorem{defn}[thm]{Definition}

\newenvironment*{proof}{\em Proof.}{$\Box$}

\begin{document}

\title{Functional Convergence of Linear Processes \\ with Heavy-Tailed Innovations\thanks{Research of R. M. Balan was supported by a
grant from NSERC of Canada}
}

\author{Raluca M. Balan  \and
        Adam Jakubowski \and Sana Louhichi
}

\date{October 12, 2014}

\maketitle

\begin{abstract}
We study convergence in law of partial sums of linear processes with heavy-tailed innovations. In the case of summable coefficients necessary and sufficient conditions for the finite dimensional convergence  to an  $\alpha$-stable L\'evy Motion are given. The conditions
lead to new, tractable sufficient conditions in the case $\alpha \leq 1$.
In the functional setting we complement the existing results on $M_1$-convergence,
obtained for linear processes with nonnegative coefficients by Avram and Taqqu (1992) and improved by Louhichi and Rio (2011), by proving that in the general setting partial sums of linear processes are convergent on the Skorokhod space equipped with the $S$ topology, introduced by Jakubowski (1997). \\[1mm]
\end{abstract}


\section{Introduction and announcement of results}
\label{intro}
Let $\{Y_j\}_{j\in\GZ}$ be a sequence of independent and identically distributed random variables. By a
{\em linear process built on innovations $\{Y_j\}$} we mean a stochastic process
\begin{equation}
\label{aje1}
 X_i=\sum_{j \in \GZ}c_j Y_{i-j}, \quad i \in \GZ,
\end{equation}
where the constants $\{c_j\}_{j\in\GZ}$ are such that the above series is  $\bP$-a.s. convergent.
Clearly, in non-trivial cases such a process is dependent, stationary  and due to the simple linear structure  many of its distributional characteristics can be easily computed (provided they exist). This refers not only to the expectation or the covariances, but also to more involved quantities, like constants for regularly varying tails (see
e.g. \cite{MS00} for discussion) or mixing coefficients (see e.g. \cite{Dou94} for discussion).

There exists a huge literature devoted to applications of linear processes in
statistical analysis and modeling of time series. We refer to the popular textbook
\cite{BrDa96}  as an excellent introduction to the topic.

Here we would like to stress  only two particular features of linear processes.

First, linear processes provide a natural illustration for phenomena of
{\em local} (or {\em weak}) dependence and  {\em long-range} dependence. The most striking results go back to Davydov (\cite{Dav70}), who obtained a rescaled fractional Brownian motion as a functional weak limit for suitable
normalized partial sums  of  $\{X_i\}$'s.

Another important property of linear processes
is the propagation of big values. Suppose that {\em some} random variable $Y_{j_0}$ takes a big value,
then this big value is propagated along the sequence $X_i$ (everywhere, where $Y_{j_0}$ is taken with a big coefficient $c_{i-j_0}$).
Thus linear processes form the simplest model for phenomena of clustering of big values, what is important
in models of insurance (see e.g. \cite{MS00}).

In the present paper we shall deal with heavy-tailed innovations. More precisely, we shall assume
that the law of $Y_i$ belongs to the domain of strict attraction
of a non-degenerate strictly $\alpha$-stable law $\mu_{\alpha}$,  i.e.
\begin{equation}\label{aje2}
Z_n=\frac{1}{a_n}\sum_{i=1}^{n}Y_i \indist Z,
\end{equation}
 where $Z \sim \mu_{\alpha}$.

Let us observe that by the Skorokhod theorem (\cite{Sko57}) we also have
\begin{equation}\label{aje3}
Z_n(t)=\frac{1}{a_n}\sum_{i=1}^{[nt]}Y_i \indist Z(t),
\end{equation}
where $\{Z(t)\}$ is the stable  L\'evy process with $Z(1) \sim \mu_{\alpha}$,
and the convergence holds on the Skorokhod space
$\GD([0,1])$, equipped with  the Skorokhod $J_1$  topology.

Recall, that if the variance of $Z$ is {\em infinite}, then (\ref{aje2})
implies the existence of  $\alpha \in (0,2)$ such that
\begin{equation}\label{aje4}
 \bP( |Y_j| > x) = x^{-\alpha} h(x), \ x >0,
\end{equation}
where  $h$ is a function that varies slowly at $x=+\infty$, and also
\begin{equation}
\label{aje5}
\lim_{x \to \infty}\frac{\bP(Y_j>x)}{\bP(|Y_j|>x)}=p \quad \mbox{and} \quad
\lim_{x \to \infty}\frac{\bP(Y_j<-x)}{\bP(|Y_j|>x)}=q, \quad p+q = 1.
\end{equation}
The norming constants $a_n$ in (\ref{aje3}) must satisfy
\begin{equation}\label{aje6}
n \bP(|Y_j| > a_n) = \frac{ n h(a_n)}{a_n^{\alpha}} \to C >0,
\end{equation}
hence are necessarily of the form $a_n = n^{1/\alpha} g(n^{1/\alpha})$, where the slowly varying function $g(x)$ is the de Bruijn conjugate of $\big(C/h(x)\big)^{1/\alpha}$ (see \cite{BGT87}).
Moreover, if $\alpha > 1$, then $\bE Y_j = 0$
and if $\alpha = 1$, then $p = q$ in (\ref{aje5}).

Conversely, conditions (\ref{aje4}), (\ref{aje5}) and
\begin{eqnarray} \label{aje7}
\bE \big[Y_j\big] = 0, &\text{ if }& \alpha > 1,\\ \label{aje8}
\{Y_j\} \text{ are symmetric,} &\text{ if }& \alpha = 1,
\end{eqnarray}
imply (\ref{aje3}).

If $a_n$ is chosen to satisfy (\ref{aje6}) with $C=1$, then  $\mu_{\alpha}$ is given by the characteristic function
\begin{equation}\label{aje9}
\hat{\mu}(\theta) = \begin{cases}
\exp\Big( \int_{\GR^1} (e^{i\theta x} - 1) f_{\alpha, p, q} (x) \,dx\Big) & \text{if
$0<\alpha<1$,}\\
\exp\Big( \int_{\GR^1} (e^{i\theta x} - 1) f_{1, 1/2, 1/2} (x) \,dx\Big)
 & \text{if $\alpha=1$,}\\
\exp\Big( \int_{\GR^1} (e^{i\theta x} - 1 - i\theta x) f_{\alpha, p, q} (x) \,dx\Big)
 &\text{if $1<\alpha<2$},
\end{cases}
\end{equation}
where
\[ f_{\alpha, p, q}(x) = \big(p\, \GI(x>0) + q\, \GI(x<0)\big) \alpha |x|^{- (1+\alpha)}. \]
We refer to \cite{Fel71} or any of contemporary monographs on limit theorems for the above basic information.

Suppose that the tails of $|Y_j|$ are regularly varying, i.e.
(\ref{aje4}) holds  for some $\alpha\in (0,2)$, and the (usual) regularity conditions
(\ref{aje7}) and (\ref{aje8}) are satisfied.
It is an observation due to Astrauskas \cite{Ast83} (in fact: a direct consequence of the
Kolmogorov Three Series Theorem - see Proposition \ref{Prop0} below) that the series (\ref{aje1}) defining the linear process
$X_i$ is $\bP$-a.s. convergent if, and only if,
\begin{equation}\label{aje10}
\sum_{j\in\GZ} |c_j|^{\alpha} h(|c_j|^{-1}) < +\infty.
\end{equation}
Given the above series is convergent we can define
\begin{equation}\label{aje11}
S_n(t)=\frac{1}{b_n} \sum_{i=1}^{[nt]}X_i,\ t\geq 0,
\end{equation}
and it is natural to ask for convergence of $S_n$'s, when $b_n$ is suitably chosen.
Astrauskas \cite{Ast83} and Kasahara \& Maejima \cite{KM88} showed that
{\em fractional stable L\'evy Motions} can appear in the limit of $S_n(t)$'s, and that some of the limiting processes can have {\em regular} or even {\em continuous}  trajectories,
while trajectories of other can be {\em unbounded on every interval}.

In the present paper we consider the important case of summable coefficients:
\begin{equation}
\label{ajsummable}
 \sum_{j\in\GZ} |c_j| < +\infty.
\end{equation}

In Section 2 we  give necessary and sufficient conditions for the finite dimensional convergence
\begin{equation}\label{ajfdd}
 S_n(t)=\frac{1}{a_n} \sum_{i=1}^{[nt]}X_i \infdd A\cdot Z(t),
\end{equation}
where the constants $a_n$ are the same as in (\ref{aje2}), $A = \sum_{j\in \GZ} c_j$ and $\{Z(t)\}$ is an $\alpha$-stable L\'evy Motion such that $Z(1) \sim Z$.
The obtained conditions lead to tractable sufficient conditions, which in case $\alpha < 1$ are new and essentially weaker than condition
\[  \sum_{j\in\GZ} |c_j|^{\beta} < +\infty, \quad \text{for some $0 < \beta < \alpha$,} \]
considered in \cite{Ast83}, \cite{DaRe85} and \cite{KM88}. See Section \ref{secsc} for details.
Notice that in the case  $A = 0$ another normalization $b_n$ is possible with a
non-degenerate limit. We refer to \cite{Pau13} for comprehensive analysis of dependence structure of infinite variance processes.

Section \ref{secfc} contains strengthening of (\ref{ajfdd}) to a functional convergence in some suitable topology on the Skorokhod space $\GD([0,1])$. Since the paper \cite{AvTa92}
it is known that in non-trivial cases (when at least two coefficients are non-zero) the convergence in the Skorokhod $J_1$ topology cannot hold. In fact none of Skorokhod's $J_1$, $J_2$, $M_1$ and $M_2$ topologies is applicable. This can be seen by analysis of the following simple example (\cite{AvTa92}, p. 488). Set $c_0 = 1, c_1 = -1$ and $c_i = 0$ if $j\neq 0,1$. Then
 $X_i = Y_i - Y_{i-1}$ and (\ref{ajfdd}) holds with $A = \sum_j c_j = 0$,
i.e.
\[ S_n(t) \inprob 0, \ t\geq 0.\]
But we see that
\[ \sup_{t\in[0,1]} S_n(t) = \max_{k\leq n} \big(Y_k - Y_0\big)/a_n\]
 converges in law to a Fr\'echet distribution. This means that {\em supremum} is not a continuous (or almost surely continuous) functional, what excludes convergence in Skorokhod's topologies in the {\em general} case.

For linear processes with {\em nonnegative} coefficients $c_i$ partial results were obtained by Avram and Taqqu \cite{AvTa92}, where convergence in the $M_1$ topology was considered. Recently these results have been improved and developed in various directions in \cite{LoRi11}
and \cite{BKS12}. We use the linear structure of processes and the established convergence in the $M_1$ topology to show that in the general case, the finite dimensional convergence  (\ref{ajfdd}) can be strengthen to convergence in the so-called $S$ topology, introduced in \cite{J97-EJP}. This is a sequential and non-metric, but fully operational topology, for which addition is {\em sequentially} continuous.

Section \ref{seccomp} is devoted to some consequences of results obtained in previous sections. We provide examples of functionals continuous in the $S$ topology. In particular we show that for every $\gamma >0$
\[ \frac{1}{n a_n^{\gamma}} \sum_{k=1}^n \Big( \sum_{i=1}^k
\big(\sum_j c_{i-j} Y_j\big) - A Y_i\Big)^{\gamma} \inprob 0.\]
We also discuss possible extensions of the theory to linear sequences built on {\em dependent} summands.

The Appendix contains technical results of independent interest. \\[2mm]

{\bf Conventions and notations.}\ Throughout the paper, in order to avoid permanent repetition of
standard assumptions and conditions we adopt the following conventions. We will say that $\{Y_j\}$'s satisfy {\em the usual conditions}
if  they are {\em independent identically distributed} and (\ref{aje4}), (\ref{aje5}), (\ref{aje7}) and
(\ref{aje8}) hold. When we write $X_i$ it is always the linear process given by (\ref{aje1}) and is well-defined, i.e. satisfies  (\ref{aje10}). Similarly the norming constants $\{a_n\}$ are defined by (\ref{aje6}) and the normalized partial sums $S_n(t)$ and $Z_n(t)$ are
given by (\ref{aje11}) with $b_n = a_n$ and (\ref{aje3}), respectively, where $Z$ is the limit in (\ref{aje2}) and $Z(t)$ is the stable L\'evy Motion such that $Z(1) \sim Z$.

\section{Convergence of finite dimensional distributions for summable coefficients}

We begin with stating the main result of this section followed by its important consequence.
\begin{thm}\label{Th1}
Let $\{Y_j\}$ be an i.i.d. sequence satisfying the usual conditions.
Suppose that
\[ \sum_j |c_j| < +\infty.\]
Then
\[
 S_n(t)=\frac{1}{a_n} \sum_{i=1}^{[nt]}X_i \infdd A\cdot Z(t),\ \text{ where $A = \sum_j c_j$},
\]
if, and only if,
\begin{equation}\label{ajcondfdd}
\begin{split}
\sum_{j=-\infty}^0 \frac{\big|d_{n,j}\big|^{\alpha}}{a_n^{\alpha}}
h\Big( \frac{a_n}{\big|d_{n,j}\big|}\Big) &\to 0,\ \text{ as $n\to\infty$,}\\
\sum_{j=n+ 1}^{\infty} \frac{\big|d_{n,j}\big|^{\alpha}}{a_n^{\alpha}}
h\Big( \frac{a_n}{\big|d_{n,j}\big|}\Big) &\to 0,\ \text{ as $n\to\infty$,}
\end{split}
\end{equation}
where
\[ d_{n,j} = \sum_{k=1-j}^{n-j}c_k, \quad n\in\GN, j\in \GZ.\]
\end{thm}

\begin{cor}\label{corext}
Under the assumptions of Theorem \ref{Th1}, define
\begin{equation}\label{eqcece}
U_i =  \sum_j |c_{i-j}| Y_j,\quad X_i^+ = \sum_j c_{i-j}^+ Y_j,\quad
X_i^- =  \sum_j c_{i-j}^- Y_j,
\end{equation}
where $c^+= c\vee 0,\ c^- = (-c) \vee 0,\ c\in\GR^1$, and  set
\begin{equation}\label{eqtete}
T_n(t) =  \frac{1}{a_n} \sum_{i=1}^{[nt]} U_i, \quad T_n^+(t) = \frac{1}{a_n} \sum_{i=1}^{[nt]} X_i^+,\quad
 T_n^-(t) = \frac{1}{a_n} \sum_{i=1}^{[nt]} X_i^-.
\end{equation}
Then
\begin{equation*}
 T_n(t) \infdd A_{|\cdot|}\cdot Z(t),\ \text{ where $A_{|\cdot|} = \sum_j |c_j|$},
\end{equation*}
implies
\begin{align*}
T_n^+(t) \infdd A_{+}\cdot Z(t),&\ \text{ where $A_{+} = \sum_j c_j^+$},\\
T_n^-(t)\infdd A_{-}\cdot Z(t),&\ \text{ where $A_{-} = \sum_j c_j^-$},\\
S_n(t) = T_n^+(t) - T_n^-(t)\infdd A \cdot Z(t),&\ \text{ where $A = \sum_j c_j$}.
\end{align*}
\end{cor}

\noindent{\em Proof of Corollary \ref{corext}.}\  \ In view of Theorem \ref{Th1} it is enough to notice that
 \[\frac{\big|d_{n,j}\big|^{\alpha}}{a_n^{\alpha}}
h\Big( \frac{a_n}{\big|d_{n,j}\big|}\Big) = \bP\Big( \Big|  \sum_{k=1-j}^{n-j}c_k\Big| \cdot |Y_j | > a_n \Big) \leq \bP\Big( \Big( \sum_{k=1-j}^{n-j}|c_k|\Big) \cdot |Y_j | > a_n \Big).\]
\mbox{}\\[3mm]

\noindent{\em Proof of Theorem \ref{Th1}.}\ \
Using Fubini's theorem, we obtain that
\begin{equation}\label{keyform}
S_n(t)=\frac{1}{a_n}\sum_{i=1}^{[nt]}\sum_{j \in \GZ}c_{i-j}Y_j
=\sum_{j\in\GZ} \frac{1}{a_n}\left(\sum_{k=1-j}^{[nt]-j}c_k\right) Y_j =\sum_{j \in \GZ} \frac{1}{a_n} d_{[nt],j} Y_j.
\end{equation}
Further, we may decompose
\begin{equation}\label{keyformplus}
\begin{split}
\sum_{j\in\GZ} \frac{1}{a_n}d_{[nt],j} Y_j = &
\sum_{j=-\infty}^0 \frac{1}{a_n}d_{[nt],j} Y_j \\
&+
\sum_{j=1}^{[nt]} \frac{1}{a_n}d_{[nt],j} Y_j \\
&+ \sum_{j=[nt]+ 1}^{\infty}  \frac{1}{a_n}d_{[nt],j} Y_j \\
= & S_n^-(t) + S_n^0(t) + S_n^+(t).
\end{split}
\end{equation}

Let us consider the partial sum process:
\[Z_n(t)=\frac{1}{a_n}\sum_{i=1}^{[nt]}Y_i,\ t \geq 0.\]
First we will show
\begin{lem}\label{L2} Under the assumptions of Theorem \ref{Th1} we have for each $t > 0$
\begin{equation}\label{ajekey1}
S_n^0(t) - A \cdot Z_n(t) \inprob 0.
\end{equation}
In particular,
 \begin{equation}\label{ajekey2}
S_n^0(t)  \indist A \cdot Z(t).
\end{equation}
\end{lem}
\noindent {\em Proof of Lemma \ref{L2}} Define
\begin{equation}\label{ajl2e1}
 V_n^0 =
\sum_{j=1}^{[nt]} \frac{\big(A - d_{[nt],j}\big)}{a_n} Y_j = A \cdot Z_n(t) - S_n^0(t).
\end{equation}
To prove that $V^0_n \inprob 0$ we apply Proposition \ref{Prop1}. We have to show that
\begin{equation}\label{ajekey3}
 \begin{split}
\sum_{j=1}^{[nt]} &\frac{\big| A - d_{[nt],j}\big|^{\alpha}}{a_n^{\alpha}} h\Big(  \frac{a_n}{\big| A - d_{[nt],j}\big|}\Big) \\
&\hspace{2cm} =  \sum_{j=1}^{[nt]} \bP\big( \big| A - d_{[nt],j}\big|\cdot |Y_j| > a_n\Big) \to 0,
\ \text{ as $n\to\infty$}.
\end{split}
\end{equation}
Since $a_n \to \infty$ and $\big| A - d_{[nt],j}\big| \leq \sum_{k\in\GZ} |c_k|$, we have
\begin{equation}
\label{eqmax}
 \max_{1\leq j \leq [nt]} \bP\big( \big| A - d_{[nt],j}\big|\cdot |Y_j| > a_n\big) \to 0.
\end{equation}
We need a simple lemma.
\begin{lem}\label{Lemsimple}
Let $\{a_{n,j}\,;\, 1 \leq j \leq n, \ n\in\GN\}$ be an array of numbers such that
\[ \max_{1\leq j\leq n} |a_{n,j}| \to 0, \ \text{ as $n\to\infty$.}
\]
Then there exists a sequence $j_n\to\infty$, $j_n = o(n)$, such that
\[ \sum_{j=1}^{j_n} |a_{n,j}| \to 0.\]
\end{lem}
\noindent {\em Proof of Lemma \ref{Lemsimple}}
For each $m\in\GN$ there exists $N_m > \max\{N_{m-1}, m^2\}$ such that for $n \geq N_m$
\[ \sum_{j=1}^{m} |a_{n,j}| < \frac{1}{m}.\]
Set  $j_n = m$, if $N_m \leq n < N_{m+1}$.
By the very definition, if $ N_m \leq n < N_{m+1}$ then
\[ \sum_{j=1}^{j_n} |a_{n,j}| < \frac{1}{m}\qquad \text{  and  }\qquad \frac{j_n}{n} \leq \frac{j_n}{N_m} = \frac{m}{N_m} \leq \frac{m}{m^2} = \frac{1}{m}.\]

By the above lemma and (\ref{eqmax}) we can find a sequence $j_n \to \infty$, $j_n = o(n)$, increasing so slowly that still
\[
\sum_{j=1}^{j_n} \bP\big( \big| A - d_{[nt],j}\big|\cdot |Y_j| > a_n\big)   + \sum_{j=[nt]-j_n + 1}^{[nt]} \bP\big( \big| A - d_{[nt],j}\big|\cdot |Y_j| > a_n\big) \to 0.
\]
For the remaining part we have
\[ \max_{j_n < j \leq [nt]- j_n} \big| A - d_{[nt],j}\big| =  \max_{j_n < j \leq [nt]- j_n} \big| A - \sum_{k=1-j}^{[nt]-j}c_k\big| = \delta_n \to 0,\]
hence for $\delta \geq \delta_n$
\[\begin{split}
\sum_{j=j_n+1}^{[nt] - j_n} \bP\big( \big| A - d_{[nt],j}\big|\cdot |Y_j| > a_n\big) &\leq \sum_{j=j_n+1}^{[nt] - j_n} \bP\Big( |\delta_n| |Y_j| > a_n\Big) \\
& \leq
\sum_{j=1}^{[nt]}\bP\Big( |\delta_n| |Y_j| > a_n\Big) \\
& \leq [nt] \frac{\delta^{\alpha}}{a_n^{\alpha}} h(a_n/\delta) \\
&= [nt] \delta^{\alpha}\frac{h(a_n)}{a_n^{\alpha}}
\frac{h(a_n/\delta)}{h(a_n)}.
\end{split}
\]
Since  $ n a_n^{-\alpha} h(a_n) = n \bP(|Y| > a_n) \to 1$ and $h$ varies slowly we have
\[ [nt] \delta^{\alpha}\frac{h(a_n)}{a_n^{\alpha}}
\frac{h(a_n/\delta)}{h(a_n)} \sim [nt] \delta^{\alpha} \frac{1}{n} \to t \delta^{\alpha},
 \text{ as $n\to\infty$}.\]
But $\delta > 0$ is arbitrary, hence we have proved (\ref{ajekey3}) and
\[  V_n^0 = A\cdot Z_n(t) - S_n^0(t) \inprob 0.\]
Since
\[ A \cdot Z_n(t) \indist A \cdot Z(t),\]
Lemma \ref{L2} follows.

In the next step we shall prove
\begin{lem}\label{L3} Under the assumptions of Theorem \ref{Th1} the following items
(i)-(iii) are equivalent.
\begin{description}
\item{\bf (i)}
\begin{equation}\label{ajeonedimone}
 S_n(1) \indist A\cdot Z(1),
\end{equation}
\item{\bf(ii)}
\begin{equation}\label{ajcor2e1}
 S_n^{-}(1) +  S_n^{+}(1) \inprob 0.
\end{equation}
\item{\bf(iii)} For every $t\in [0,1]$
\begin{equation}\label{ajeonedim}
 S_n(t) - A\cdot Z_n(t) \inprob 0.
\end{equation}
\end{description}
\end{lem}
\noindent {\em Proof of Lemma \ref{L3}}
By Lemma \ref{L2} we know that $S_n^0(1) - A\cdot Z_n(1) \inprob 0$ and
$S_n^0(1) \indist A\cdot Z(1)$.
Since $S_n(1) = S_n^-(1) + S_n^0(1) + S_n^+(1)$, (\ref{ajeonedim}) implies
(\ref{ajcor2e1}) and the latter implies (\ref{ajeonedimone}).

So let us assume  (\ref{ajeonedimone}). By regular variation of $a_n$ we have for each $t\in(0,1]$
\[ S_n(t) = \frac{1}{a_n} \sum_{i=1}^{[nt]} X_i = \frac{a_{[nt]}}{a_n} \frac{1}{a_{[nt]}}
 \sum_{i=1}^{[nt]} X_i \indist t^{1/\alpha}A\cdot Z(1) \sim A\cdot Z(t).\]
It follows that
\[ \bE \big[e^{i\theta S_n(t)}\big] = \bE \big[e^{i\theta S_n^0(t)}\big] \bE \big[e^{i\theta \big(S_n^-(t) + S_n^+(t)\big)}\big] \to \bE \big[e^{i\theta A\cdot Z(t)}\big], \ \theta \in\GR^1.\]
Since also
\[\bE \big[e^{i\theta S_n^0(t)}\big] \to \bE \big[e^{i\theta A\cdot Z(t)}\big], \ \theta \in\GR^1,\]
and $\bE \big[e^{i\theta A\cdot Z(t)}\big] \neq 0, \ \theta\in\GR^1$ (for $Z(t)$ has infinitely divisible law),
we conclude that
\[\bE \big[e^{i\theta \big(S_n^-(t) + S_n^+(t)\big)}\big] \to 1, \ \theta \in\GR^1.\]
Thus $S_n^-(t) + S_n^+(t) \inprob 0$ and by Lemma
\ref{L2} also $S_n^0(t) - A\cdot Z(t) \inprob 0$. Hence   (\ref{ajeonedim})
follows.

Let us observe that  by Proposition \ref{Prop1} (\ref{ajcor2e1}) holds if, and only if,
\begin{equation}\label{ajtwoinsteadofone}
\sum_{j=-\infty}^0 \frac{\big|d_{n,j}\big|^{\alpha}}{a_n^{\alpha}}
h\Big( \frac{a_n}{\big|d_{n,j}\big|}\Big)
+\ \sum_{j=n+ 1}^{\infty} \frac{\big|d_{n,j}\big|^{\alpha}}{a_n^{\alpha}}
h\Big( \frac{a_n}{\big|d_{n,j}\big|}\Big) \to 0,\ \text{ as $n\to\infty$,}
\end{equation}
i.e. relation (\ref{ajcondfdd}) holds. Therefore the proof of Theorem \ref{Th1} will be complete, if we can show that convergence of {\em one-dimensional} distributions implies the finite dimensional  convergence. But this is obvious in view of (\ref{ajeonedim}):
\[ \big(S_n(t_1), S_n(t_2), \ldots, S_n(t_m)\big) - A\cdot \big(Z_n(t_1), Z_n(t_2), \ldots, Z_n(t_m)\big) \inprob 0,\]
and the  finite dimensional distributions of stochastic processes $A\cdot Z_n(t)$ are convergent to those of $A\cdot Z(t)$.

\begin{rem}\label{Rem1}
Observe that for one-sided moving averages the two conditions in (\ref{ajcondfdd})
reduce to one (the expression in the other equals $0$). This is the reason we use in Theorem
\ref{Th1} two conditions replacing the single statement (\ref{ajtwoinsteadofone}).
\end{rem}

\begin{rem}\label{RemConst}
In the proof of Proposition \ref{Prop1} we used the Three Series Theorem with the level
of truncation $1$. It is well known that  any $r \in (0,+\infty)$ can be chosen as the truncation level.
Hence conditions  (\ref{ajcondfdd}) admit an equivalent reformulation in the ``$r$-form''
\begin{equation*}
\begin{split}
\sum_{j=-\infty}^0 \frac{\big|d_{n,j}\big|^{\alpha}}{a_n^{\alpha}}
h\Big( \frac{r \cdot a_n}{\big|d_{n,j}\big|}\Big) &\to 0,\ \text{ as $n\to\infty$.}\\
\sum_{j=n+ 1}^{\infty} \frac{\big|d_{n,j}\big|^{\alpha}}{a_n^{\alpha}}
h\Big( \frac{r \cdot a_n}{\big|d_{n,j}\big|}\Big) &\to 0,\ \text{ as $n\to\infty$.}
\end{split}
\end{equation*}
\end{rem}

\section{Functional convergence}\label{secfc}

\subsection{Convergence in the $M_1$ topology}

As outlined in Introduction (see also Section \ref{seccompl} below), the convergence of finite dimensional distributions of linear processes built on heavy-tailed innovations cannot be, in general, strengthened to functional convergence in any of Skorokhod's topologies $J_1, J_2, M_1, M_2$.

The general linear process $\{X_i\}$ can be, however, represented as a difference of linear processes with non-negative coefficients. Let us recall the notation introduced in Corollary \ref{corext}:
\begin{align*}
 X_i^+ =& \sum_j c_{i-j}^+ Y_j,\qquad T_n^+(t) = \frac{1}{a_n} \sum_{i=1}^{[nt]} X_i^+,\\
\quad X_i^- =&  \sum_j c_{i-j}^- Y_j,\qquad T_n^-(t) = \frac{1}{a_n} \sum_{i=1}^{[nt]} X_i^-.
\end{align*}

Notice, that in general $X_i^{\pm}(\omega)$ {\em is not} equal to $\big(X_i(\omega)\big)^{\pm}$ and  that we have
\begin{equation}\label{ajs32}
S_n(t) = T_n^+(t) - T_n^-(t).
\end{equation}

The point is that both $T_n^+(t)$ and $T_n^-(t)$ are partial sums of {\em associated} sequences in the sense of \cite{EPW67} (see e.g. \cite{BuSh07} for the contemporary theory)
and thus exhibit much more regularity.

Theorem 1 of Louhichi and Rio \cite{LoRi11} can be specified to the case of linear processes considered in our paper in the following way.

\begin{prop}\label{PropLR}
Let the innovation sequence $\{Y_j\}$ satisfies the usual conditions. Let
\begin{equation}\label{ajs3plus}
c_j \geq 0, j\in\GZ, \text{ and } \sum_j c_j < +\infty.
\end{equation}
If the linear process $\{X_i\}$ is well-defined and
\[ S_n(t)  \infdd A\cdot Z(t),\]
then also functionally
\[ S_n \indist A\cdot Z\]
on the Skorokhod space $\GD([0,1])$ equipped with the $M_1$ topology.
\end{prop}

\begin{rem}\label{Rems31}
The first result of this type was obtained by Avram and Taqqu \cite{AvTa92}. They required
however more regularity on coefficients (e.g. monotonicity of $\{c_j\}_{j\geq1}$ and
$\{c_{-j}\}_{j\geq 1}$).
\end{rem}

\subsection{$M_1$-convergence implies $S$-convergence}

Let us turn to linear processes with coefficients of arbitrary sign. Given decomposition (\ref{ajs32}) and Proposition \ref{PropLR} the strategy is now clear: choose any {\em linear} topology $\tau$ on $\GD([0,1])$ which is {\em coarser} than $M_1$,
then
\[ S_n(t)  \infdd A\cdot Z(t),\]
should imply
\[ S_n \indist A\cdot Z\]
on the Skorokhod space $\GD([0,1])$ equipped with the topology $\tau$. Since convergence of c\`adl\`ag functions in the $M_1$ topology is bounded and implies pointwise convergence outside of a countable set,
there are plenty of such topologies. For instance any space of the form $L^{p}\big([0,1],\mu\big)$, where $p\in [0,\infty)$ and $\mu$ is an {\em atomless} finite measure on $[0,1]$, is suitable.  The point is to choose the {\em finest} among linear topologies with required properties, for we want to have the maximal family of continuous functionals on $\GD([0,1])$,

Although we are not able to identify such an ``ideal" topology,  we  believe that this distinguished position belongs to the $S$ topology, introduced in \cite{J97-EJP}.  This is a non-metric sequential topology, with sequentially continuous addition, which is stronger than any of mentioned above $L^p(\mu)$ spaces and is functional in the sense it has the following classic property (see Theorem 3.5 of \cite{J97-EJP}).

\begin{prop}\label{Prop3}
Let $\GQ \subset [0,1]$ be dense, $1 \in \GQ$. Suppose that for each finite subset
$\GQ_0 = \{  q_1 < q_2 < \ldots < q_m \} \subset \GQ$ we have as $n\to \infty$
\[ (X_n(q_1), X_n(q_2), \ldots, X_n(q_m)) \indist
(X_0(q_1), X_0(q_2), \ldots, X_0(q_m)),\]
where $X_0$ is a stochastic process with trajectories in $\GD[0,1])$.
If $\{X_n\}$ is uniformly $S$-tight, then
\[ X_n \indist X_0,\]
on the Skorokhod space $\GD([0,1])$ equipped with the  $S$ topology.
\end{prop}

For readers familiar with the limit theory for stochastic processes the above property may seem obvious. But it is trivial only for processes with continuous trajectories. It is not trivial even  in the case of the Skorokhod $J_1$ topology, since the point evaluations
\[ \pi_t : \GD([0,1]) \to \GR^1,\ \pi_t(x) = x(t),\]
can be $J_1$-discontinuous at some $x\in \GD([0,1])$ (see \cite{Top69} for the result corresponding to Proposition \ref{Prop3}).
In the  $S$ topology
 the point evaluations are {\em nowhere} continuous (see \cite{J97-EJP}, p. 11).
Nevertheless Proposition \ref{Prop3} holds for the $S$ topology, while it {\em does not hold}
for the linear metric spaces $L^p(\mu)$ considered above. It follows that the $S$ topology is suitable for the needs of limit theory for stochastic processes. It admits even such efficient tools like the a.s Skorokhod representation for subsequences \cite{J97-TPA}. On the other hand, since $\GD([0,1])$ equipped with $S$ is non-metric and sequential, many of apparently standard reasonings require special tools and careful analysis. This will be seen below.

Before we define the $S$ topology we need some notation.
 Let $\GV([0,1]) \subset \GD([0,1])$ be the space of (regularized) functions of finite variation on $[0,1]$,
equipped with the norm of total variation $\|v\| = \|v\|(1)$, where
\[ \|v\|(t) = \sup \Big\{|v(0)| + \sum_{i=1}^{m} |v(t_i) - v(t_{i-1})|\Big\},\]
and the supremum is taken over all finite partitions $0=t_0 < t_1 < \ldots < t_m = t$. Since $\GV([0,1])$ can be identified with a dual of $(\GC([0,1]),\|\cdot\|_{\infty})$, we have on it the
weak-$*$ topology. We shall write $v_n \Rightarrow v_0$ if for every $f \in \GC([0,1])$
\[
\int_{[0,1]} f(t) dv_n(t) \to \int_{[0,1]} f(t) dv_0(t).
\]

\begin{defn} {\bf ($S$-convergence and the $S$ topology)}
We shall say that $x_n$ $S$-converges to $x_0$ (in short $x_n \to_S x_0$)
if for every $\varepsilon > 0$ one can
find elements $v_{n,\varepsilon}\in \GV([0,1])$, $n=0,1,2,\ldots $ which are
$\varepsilon$-uniformly close to $x_n$'s and weakly-$*$ convergent:
\begin{eqnarray}
\|x_n - v_{n,\varepsilon}\|_{\infty} \leq \varepsilon,&&\ n = 0, 1, 2,
\ldots,\\
v_{n,\varepsilon} \Rightarrow v_{0,\varepsilon},&& \text{as $n\to\infty$}.
\end{eqnarray}
The $S$  topology is the sequential topology determined by
the $S$-convergence.
\end{defn}

\begin{rem}\label{RemonS} This definition was given in \cite{J97-EJP} and we refer to this paper for detailed derivation of basic properties  of $S$-convergence and construction of the $S$ topology, as well as for instruction how to effectively operate with $S$. Here we shall stress only that the  $S$ topology emerges naturally in the context of the following {\em  criteria of compactness}, which will be used in the sequel.
\end{rem}

\begin{prop}[2.7 in \cite{J97-EJP}]\label{ajcritcomp}
For $\eta > 0$, let $N_{\eta}(x)$ be the {\em number of $\eta$-oscillations} of the function $x \in \GD([0,1])$, i.e. the largest integer $N \geq 1$, for which there exist some points
\[0 \leq t_1<t_2 \leq t_3<t_4 \leq \ldots \leq t_{2N-1}<t_{2N} \leq 1,\]
such that
\[|x(t_{2k})-x(t_{2k-1})|>\eta \quad \mbox{for all} \ k=1,\ldots,N.\]

Let $\cK\subset \GD$.  Assume that
\begin{align}\label{ajsup}
\sup_{x\in \cK} \|x\|_{\infty} &< +\infty,\\
\sup_{x\in \cK} N_{\eta}(x) &< +\infty, \text{ for each $\eta > 0$.}\label{ajosc}
\end{align}
Then from any sequence $\{x_n\} \subset \cK$ one can extract a subsequence $\{x_{n_k}\}$ and find $x_0 \in \GD([0,1])$ such that $x_{n_k}
\ines x_0$.

Conversely, if $\cK\subset \GD([0,1])$ is relatively compact with respect to $\inessh$, then it satisfies both (\ref{ajsup}) and (\ref{ajosc}).
\end{prop}

\begin{cor}[2.14 in \cite{J97-EJP}]\label{ajcritcor}
Let $\GQ \subset [0,1]$, $1\in\GQ$, be dense. Suppose that $\{x_n\}\subset \GD([0,1])$ is relatively $S$-compact and as $n\to \infty$
\[
x_n(q) \to x_0(q),\ q \in \GQ.\]
Then $x_n \to x_0$ in $S$.
\end{cor}
\begin{rem}\label{RemKK}
The  $S$  topology is {\em sequential}, i.e. it is generated by the convergence $\inessh$.
By the Kantorovich-Kisy\'nski recipe \cite{Kis60} $x_n \to x_0$ in  $S$ topology if, and only if, in each subsequence $\{x_{n_k}\}$ one can find a further subsequence $x_{n_{k_l}} \ines x_0$. This is the same story as with a.s. convergence and convergence in probability of random variables.
\end{rem}

According to our strategy, we are going to prove that
Skorokhod's $M_1$-topology is stronger than the $S$ topology or, equivalently, that $x_n \inemone x_0$ implies  $x_n \ines x_0$. We refer the reader to the original Skorohod's article \cite{Sko56} for the definition of the $M_1$ topology, as well as to Chapter 12 of \cite{Whi02} for a comprehensive account of properties of this topology.

The $M_1$-convergence can be described using a suitable modulus of continuity. We define for $x\in\GD([0,1])$ and $\delta > 0$
\begin{equation}\label{eqomega}
w^{M_1}(x,\delta):=\sup_{0 \vee (t_2-\delta) \leq t_1<t_2<t_3 \leq 1 \wedge (t_2+\delta)}H\big(x(t_1),x(t_2),x(t_3)\big),
\end{equation}
where $H(a,b,c)$ is the distance between $b$ and the interval with endpoints $a$ and $c$:
\[H(a,b,c)=(a\wedge c-a \wedge c \wedge b) \vee (a \vee c \vee b-a \vee c).\]

\begin{prop}[2.4.1 of \cite{Sko56}]
\label{charact-M1-conv}
Let $(x_n)_{n \geq 1}$ and $x_0$ be arbitrary elements in $\GD([0,1])$. Then
\[x_n \inemone x_0\]
if, and only if, for some dense subset $\GQ \subset [0,1]$ containing $0$ and $1$,
\begin{equation}\label{ajemonefdd}
x_n(t) \to x(t), \ t\in \GQ,
\end{equation}
and
\begin{equation}\label{ajemonemod}
\lim_{\delta \to 0}\limsup_{n \to \infty}w^{M_1}(x_n,\delta)=0.
\end{equation}
In particular, if $x_n \inemone x_0$, then
\[x_n(t) \to x_0(t)\]
for $t=1$ and at every point of continuity of $x_0$.
\end{prop}

\begin{lem}
For any $a,b,c,d \in \GR^1$
\[ |a - b| \leq |c-d| + H(c,a,d) + H(c,b,d).\]
\end{lem}
\begin{proof}
If $c \leq a \leq b \leq d$, then $b-a \leq d - c  = d-c + H(c,a,d) + H(c,b,d)$. If $a \leq c \leq b \leq d$ then $b - a = b-c + c - a \leq  d- c + H(c,a,d) = d - c + H(c,a,d) + H(c,b,d)$. If $a \leq c \leq d \leq b$ then $b - a = b - d + d - c + c -a = H(c,b,d) + d -c + H(c,a,b)$. If $ a \leq b \leq c \leq d$, then $b - a \leq |b - c| + |c-a| = H(c,b,d) + H(c,a,d) \leq   H(c,b,d) + H(c,a,d) + d - c$.
The other cases can be reduced to the considered above.
\end{proof}
\begin{cor}
\label{lem-m1-bound}
Let $x \in \GD([0,1])$. For any $0 \leq s \leq u <v \leq t \leq 1$,
\[|x(u)-x(v)| \leq |x(s)-x(t)|+H(x(s),x(u),x(t))+H(x(s),x(v),x(t)).\]
\end{cor}

\begin{lem}
\label{lem-m1-number}
Let $x \in \GD([0,1])$. For $0 \leq s<t \leq 1$, define
\[\beta=\sup_{s \leq u<v<w \leq t}H(x(u),x(v),x(w)).\]
If $\eta>2\beta$ then
\[N_{\eta}(x;[s,t]) \leq \frac{|x(t)-x(s)|+\beta}{\eta-\beta},\]
where $N_{\eta}(x;[s,t])$ denotes the number of $\eta$-oscillations of $x$ in the interval $[s,t]$.
\end{lem}
\begin{proof} Let $s \leq t_1<t_2 \leq t_3<t_4 \leq \ldots \leq t_{2N-1}<t_{2N} \leq t$ be such that
\[|x(t_{2k})-x(t_{2k-1})|>\eta \quad \text{ for all } \ k=1,\ldots,N.\]
Assume first that $x(t_2)-x(t_1)>\eta$. We claim that
\[x(t_3)\geq x(t_2)-\beta \quad \mbox{and} \quad  \quad x(t_4)-x(t_3)>\eta.\]
To see this, suppose that $x(t_3)<x(t_2)-\beta$. Then the distance between $x(t_2)$ and the interval with endpoints $x(t_1)$ and $x(t_3)$ is greater than $\beta$, which is a contradiction. Hence $x(t_3)\geq x(t_2)-\beta$. On the other hand, if we assume that $x(t_4)-x(t_3)<-\eta$, we obtain that
\[x(t_1)=x(t_1)-x(t_2)+x(t_2)-x(t_3)+x(t_3)<-\eta+\beta+x(t_3)<x(t_3)-\beta,\]
which means that the distance between $x(t_3)$ and the interval with endpoints $x(t_1)$ and $x(t_4)$ is greater than $\beta$, again a contradiction.

Repeating this argument, we infer that:
\[x(t_{2k})-x(t_{2k-1})>\eta, \quad \text{ for all } k=1, \ldots,N\]
and
\[x(t_{2k+1})-x(t_{2k})>-\beta \quad \text{ for all } k=1, \ldots, N-1.\]
Taking the sum of these inequalities, we conclude that:
\begin{equation}
\label{lemB-step1}
x(t_{2N})-x(t_1)>N\eta-(N-1)\beta=N(\eta-\beta)+\beta.
\end{equation}

On the other hand, by Corollary \ref{lem-m1-bound}, we have:
\begin{equation}
\label{lemB-step2}
|x(t_{2N})-x(t_1)|\leq |x(t)-x(s)|+2\beta.
\end{equation}

Combining (\ref{lemB-step1}) and (\ref{lemB-step2}), we obtain that
\[N \leq \frac{|x(t)-x(s)|+\beta}{\eta-\beta},\]
which is the desired upper bound.

Assuming that $x(t_2)-x(t_1) < -\eta$ we come in a similar way to the inequality
\[ x(t_{2N})-x(t_1) < -N\eta +(N-1)\beta=- N(\eta-\beta)- \beta\]
or
\[ |x(t_{2N})-x(t_1)| \>  N(\eta-\beta) + \beta.\]
This again allows us to use Corollary  \ref{lem-m1-bound} and gives the desired bound for $N$
\end{proof}

The following result was stated without proof in \cite{J97-EJP}. A short proof can be given using Skorohod's criterion 2.2.11 (page 267 of \cite{Sko56}) for the $M_1$-convergence, expressed in terms of the number of upcrossings. This proof has a clear disadvantage: it refers to an equivalent definition of the $M_1$-convergence, but the equivalence of both definitions was not proved in Skorokhod's paper. In the present article we give a complete proof.

\begin{thm}
\label{S-weaker-M1}
The $S$ topology  is weaker than the $M_1$ topology (and hence, weaker than the $J_1$ topology).
Consequently, a set $A \subset \GD([0,1])$ which is relatively $M_1$-compact is also relatively $S$-compact.
\end{thm}

\begin{proof} Let $x_n \inemone x_0$. By Proposition \ref{charact-M1-conv}
\[ x_n(t) \to x_0(t),\]
on the dense set of points of continuity of $x_0$ and for $t=1$. Suppose we know
that $\cK  =\{x_n\}$ satisfies conditions (\ref{ajsup}) and (\ref{ajosc}).
Then by Proposition \ref{ajcritcomp} $\{x_n\}$
is relatively $S$-compact and by Corollary
\ref{ajcritcor}
$x_n \to x_0$ in $S$.
Thus it remains to check conditions
\begin{align}\label{ajsup1}
K_{sup} =\sup_n \|x_n\|_{\infty} &< +\infty,\\
K_{\eta} = \sup_n N_{\eta}(x_n) &< \infty, \ \eta > 0.\label{ajosc1}
\end{align}

First suppose that $x_0(1-) = x_0(1)$. Then
$\GD([0,1]) \ni x \mapsto \|x\|_{\infty}$
is $M_1$-continuous at $x_0$. Consequently, $x_n \inemone x_0$ implies $\|x_n\|_{\infty} \to
\|x_0\|_{\infty}$ and (\ref{ajsup1}) follows.

If $x_0(1-) \neq x_0(1)$ we have to proceed a bit more carefully. Consider (\ref{ajemonemod}) and take $\delta > 0$ and $n_0$ such that
$w(x_n,\delta) \leq 1, n \geq n_0$. Find  $ t_0 \in (1-\delta, 1)$ which is a point of continuity of $x_0$. Then
 \[ \sup_{t\in [0,t_0]} |x_n(t)| \to \sup_{t\in [0,t_0]} |x_0(t)|, \]
hence $\sup_n \sup_{t\in [0,t_0]} |x_n(t)| < +\infty$.
We also know that $x_n(t_0) \to x_0(t_0)$ and $x_n(1)
\to x_0(1)$. Choose $n\in\GN$ and $u\in (t_0, 1)$.
By the very definition of the modulus $H$
\[ \begin{split}
|x_n(u)| &\leq |x_n(t_0)| +  |x_n(1)| + H\big(x_n(t_0), x_n(u), x_n(1)\big)\\
& \leq \sup_n |x_n(t_0)| +  \sup_n|x_n(1)| + 1, \quad n \geq n_0.
\end{split}
\]
It follows that also
\[\sup_n \sup_{t\in (t_0,1]} |x_n(t)| < +\infty,\]
and so (\ref{ajsup1}) holds.

In order to prove (\ref{ajosc1}) choose $\eta>0$ and $0<\ve < \eta/2$. By Proposition \ref{charact-M1-conv}, there exist some $\delta > 0$ and an integer $n_0 \geq 1$ such that
$w^{M_1}(x_n,\delta) < \ve, \ n \geq n_0$.
Next we find a partition $0 = t_0 < t_1 < \ldots <  t_{M} = 1$ consisting of points of continuity of $x_0$ and such that
\[ t_{j+1} - t_j < \delta, \ j= 0, 1,\ldots,M-1.\]
Again by Proposition \ref{charact-M1-conv}, there exists an integer $n_1 \geq n_0$ such that for any $n \geq n_1$
\begin{equation}
\label{ineq-xn-x}| x_n(t_j) - x(t_j)| < \ve, \ j = 0,1,\ldots, M.
\end{equation}

Fix an integer $n \geq n_1$. Suppose that $N_{\eta}(x_n) \geq N$, i.e. there exist some points
\begin{equation}\label{oscfigures}
 0\leq s_1 < s_2 \leq s_3 < s_4 \leq \ldots \leq s_{2N-1} < s_{2N} \leq 1,
\end{equation}
such that
\begin{equation}\label{eaj3}
| x_n(s_{2k}) - x_n(s_{2k-1})| > \eta, \quad \mbox{for all} \ k=1,2, \ldots, N.
\end{equation}
The proof of (\ref{ajosc1}) will be complete once we estimate the number $N$ by a constant independent of $n$.

The $\eta$-oscillations of $x_n$ determined by (\ref{oscfigures}) can be divided into two (disjoint) groups. The first group (Group 1) contains the oscillations for which the corresponding interval $[s_{2k-1},s_{2k})$ contains at least one point $t_{j'}$. Since the number of points $t_j$ is $M$,
\begin{equation}
\label{group1}
\mbox{the number of oscillations in Group 1 is at most $M$}.
\end{equation}
In the second group (Group 2), we have those oscillations for which the corresponding interval $[s_{2k-1},s_{2k})$ contains no point $t_{j}$, i.e.
\begin{equation}\label{eaj4}
t_j \leq s_{2k-1} < s_{2k} \leq  t_{j+1} \quad \mbox{for some} \ j=0,1, \ldots, M-1.
\end{equation}

We now use Lemma \ref{lem-m1-number} in each of intervals $[t_j, t_{j+1}]$, $j=0,1,\ldots, m$.
Note that
\[\beta_{n,j}:=\sup_{t_{j} \leq u<v<w \leq t_{j+1}}H\big(x_n(u), x_n(v),x_n(w) \big) \leq w^{M_1}(x_n,\delta)<\ve,\]
hence,
\[N_{\eta}(x_n,[t_j,t_{j+1}]) \leq \frac{|x_n(t_{j+1})-x_n(t_{j})|+\beta_{n,j}}{\eta-\beta_{n,j}} <
\frac{ 2 K_{sup} + \varepsilon}{\eta-\varepsilon}.\]
Since there are $M$ intervals of the form $[t_j,t_{j+1}]$, we conclude that
\begin{equation}
\label{group2}
\mbox{the number of oscillations in Group 2 is at most}
\ M \cdot \frac{ 2 K_{sup} + \varepsilon}{\eta-\varepsilon}
\end{equation}
Summing (\ref{group1}) and (\ref{group2}), we obtain that
\[N \leq   M\left(1+ \frac{ 2 K_{sup} + \varepsilon}{\eta-\varepsilon}\right)=M\frac{2K_{sup}+\eta}{\eta-\ve},\]
which does not depend on $n$. Theorem \ref{S-weaker-M1} follows.
\end{proof}

For the sake of  completeness, we provide also a typical example of a sequence $(x_n)_{n \geq 1}$ in $\GD[0,1]$ which is $S$-convergent, but does not converge in the $M_1$ topology.

\begin{exa}\label{typical}
Let $x=0$ and
$$x_n(t)=1_{[1/2-1/n,1]}(t)-1_{[1/2+1/n,1]}(t)=\left\{
\begin{array}{ll} 1 & \mbox{if $\frac{1}{2}-\frac{1}{n} \leq t <\frac{1}{2}+\frac{1}{n}$} \\
0 & \mbox{otherwise}
\end{array} \right.$$
Then $x_n \ines x$. To see this, we take $v_{n,\ve}=x_{n}$. Then $v_{n,\ve} \Rightarrow v_{\ve}=0$ since for any $f \in C[0,1]$,
$$\int_{0}^{1}f(t)dv_n(t) =f\left(\frac{1}{2}-\frac{1}{n}\right)-f\left(\frac{1}{2}+\frac{1}{n}\right) \to 0.$$
The fact that $(x_n)_{n \geq 1}$ cannot converge in $M_1$ follows by Proposition \ref{charact-M1-conv} since if $t_1<\frac{1}{2}-\frac{1}{n}<t_2<\frac{1}{2}+\frac{1}{n}<t_3$, then $H\big(x_n(t_1),x_n(t_2), x_n(t_3)\big)=1$.
\end{exa}

\subsection{Convergence in distribution in the $S$ topology}

Now we are ready to specify results on functional convergence of stochastic processes in the $S$ topology, which are suitable for
needs of linear processes. They follow directly from Proposition \ref{ajcritcomp} and Proposition \ref{Prop3}.

\begin{prop}[3.1 in \cite{J97-EJP}]\label{PropT}
A family $\{X_{\gamma}\}_{\gamma\in \Gamma}$ of stochastic processes with trajectories in $\GD([0,1])$
is uniformly $S$-tight if, and only if, the families of random variables
$\{\|X_{\gamma}\|_{\infty}\}_{\gamma\in\Gamma}$
and
$\{N_{\eta}(X_{\gamma})\}_{\gamma\in\Gamma}$, $\eta > 0$, are uniformly tight.
\end{prop}

\begin{prop}\label{PropFin}
Let $\{X_n\}_{n\geq 0}$ and $\{Y_n\}_{n\geq 0}$
be two sequences of stochastic processes with trajectories in $\GD([0,1])$ such that as $n\to\infty$
\[\begin{split}
\big(X_n(q_1) &+ Y_n(q_1), X_n(q_2) + Y_n(q_2), \ldots, X_n(q_k) +Y_n(q_k)\big) \\
&\indist \big(X_0(q_1) + Y_0(q_1), X_0(q_2) + Y_0(q_2)\ldots, X_0(q_k) +Y_0(q_k)\big),
\end{split}
\]
for each subset $\GQ_0 = \{ 0 \leq q_1 < q_2 < \ldots < q_k \}$ of a dense set $\GQ\subset [0,1]$,
$1 \in \GQ$.

If $\{X_n\}$ and $\{Y_n\}$ are uniformly $S$-tight, then
\[X_n+Y_n \indist X_0+Y_0\]
on the Skorokhod space $\GD([0,1])$ equipped with the $S$ topology.
\end{prop}
\noindent{\em Proof of Proposition \ref{PropFin}}\ \
According to Proposition \ref{Prop3}, it is enough to establish the uniform $S$-tightness of  $X_n + Y_n$.
This follows immediately from Proposition \ref{PropT} and from the inequalities
$\|x+y\|_{\infty} \leq \|x\|_{\infty}+\|y\|_{\infty}$
and
\[N_{\eta}(x+y)\leq N_{\eta/2}(x)+N_{\eta/2}(y),\]
valid for arbitrary functions $x,y \in \GD[0,1]$ and $\eta > 0$.

\begin{rem}\label{RemSum} In linear topological spaces the algebraic sum $\cK_1 + \cK_2 = \{x_1 + x_2 \,;\,
x_1 \in \cK_1, x_2 \in \cK_2\}$ of {\em compact } sets $\cK_1$ and $\cK_2$ {\em is compact}. It follows directly from the
continuity of the operation of addition and trivializes the proof of uniform tightness of sum of uniformly tight random elements.  In $\GD([0,1])$ equipped with $S$ we are, however, able to prove that the addition is only sequentially continuous, i.e. if $x_n \ines x_0$ and $y_n \ines y_0$, then $x_n + y_n \ines x_0 + y_0$. In general it does not imply continuity (see \cite{J97-EJP}, p. 18, for detailed discussion). Sequential continuity gives  a weaker property:  the sum $\cK_1 + \cK_2$ of relatively $S$-compact $\cK_1$ and $\cK_2$ is relatively $S$-compact. For the uniform tightness purposes we also need that the $S$-closure of
$\cK_1 + \cK_2$ is again relatively $S$-compact. This
is guaranteed by the lower-semicontinuity in $S$ of $\|\cdot \|_{\infty}$ and $N_{\eta}$ (see \cite{J97-EJP}, Corollary 2.10).
\end{rem}

\subsection{The main result}

\begin{thm}\label{ThMain}
Let $\{Y_j\}$ be an i.i.d. sequence satisfying the usual conditions and $\sum_j|c_j| < +\infty$.
Let $S_n(t)$ be defined by (\ref{aje11}) and $T_n(t)$ by (\ref{eqtete}).
Then
\begin{equation*}
 T_n(t) \infdd A_{|\cdot|}\cdot Z(t),\ \text{ where $A_{|\cdot|} = \sum_j |c_j|$},
\end{equation*}
implies
\[S_n \indist A\cdot Z, \ \text{ where $A = \sum_j c_j$},\]
 on the Skorokhod space $\GD([0,1])$ equipped with the $S$  topology.
\end{thm}
\begin{proof}
By Corollary \ref{corext}
\begin{equation*}
T_n^+(t) = \frac{1}{a_n} \sum_{i=1}^{[nt]} X_i^+ \infdd A_+ \cdot Z(t),\ \
 T_n^-(t) = \frac{1}{a_n} \sum_{i=1}^{[nt]} X_i^-  \infdd  A_-\cdot Z(t),
\end{equation*}
where $A_+ = \sum_{i\in\GZ} c_i^+$ and $A_- = \sum_{i\in\GZ} c_i^-$.
It follows from Proposition \ref{PropLR} that $T_n^+ \indist A_+\cdot Z$ on
$\GD([0,1])$ equipped with the  $M_1$ topology. A similar result holds for $T_n^-$. Since the law of every c\`adl\`ag process is $M_1$-tight,   Le Cam's theorem \cite{LeCa57} (see also Theorem 8 in Appendix III of \cite{Bill68}) guarantees that both sequences $\{T_n^+\}$ and $\{T_n^-\}$ are uniformly $M_1$-tight. By Theorem \ref{S-weaker-M1} we obtain the uniform $S$-tightness of both $\{T_n^+\}$ and $\{T_n^-\}$.
Again by Corollary \ref{corext}
\[ S_n(t) = T_n^+(t) - T_n^-(t) \infdd A \cdot Z(t).\]
Now a direct application of Proposition \ref{PropFin} completes the proof of the theorem.
\end{proof}\label{Sec3}

\section{Discussion of sufficient conditions}

\label{secsc}
Conditions (\ref{ajcondfdd}) do not look tractable.  In what follows we shall provide three types of checkable sufficient conditions. In both cases the following slight simplification (\ref{ajcondfddsimp}) of  (\ref{ajcondfdd}) will be useful.
 As in proof of Lemma \ref{L2}, we can find a sequence $j_n \to\infty$,
$j_n = o(n)$, such that
\[
\begin{split}
\sum_{j=-j_n + 1}^0 \frac{\big|d_{n,j}\big|^{\alpha}}{a_n^{\alpha}}
h\Big( \frac{a_n}{\big|d_{n,j}\big|}\Big) &\to 0,\ \text{ as $n\to\infty$.}\\
\sum_{j=n+ 1}^{n+j_n -1} \frac{\big|d_{n,j}\big|^{\alpha}}{a_n^{\alpha}}
h\Big( \frac{a_n}{\big|d_{n,j}\big|}\Big) &\to 0,\ \text{ as $n\to\infty$.}
\end{split}
\]
Hence it is enough to check
\begin{equation}\label{ajcondfddsimp}
\begin{split}
\sum_{j=-\infty}^{-j_n} \frac{\big|d_{n,j}\big|^{\alpha}}{a_n^{\alpha}}
h\Big( \frac{a_n}{\big|d_{n,j}\big|}\Big) &\to 0,\ \text{ as $n\to\infty$.}\\
\sum_{j=n+ j_n}^{+\infty} \frac{\big|d_{n,j}\big|^{\alpha}}{a_n^{\alpha}}
h\Big( \frac{a_n}{\big|d_{n,j}\big|}\Big) &\to 0,\ \text{ as $n\to\infty$.}
\end{split}
\end{equation}
The advantage of this form of the conditions consists in the fact that
\begin{equation}
\label{ajT3e0}\begin{split}
 \sup_{j \leq - j_n}  \big|d_{n,j}\big| &\to 0,\ \text{ as $n\to\infty$},\\
  \sup_{j \geq n + j_n}  \big|d_{n,j}\big| &\to 0,\ \text{ as $n\to\infty$}.
\end{split}
\end{equation}

We will write $\indees$ when convergence in distribution with respect to the $S$ topology takes place.

\begin{cor}\label{ThADR}
Under the assumptions of Theorem \ref{Th1}, if there exists $0 < \beta < \alpha$, $\beta \leq 1$ such that
\begin{equation}\label{ajcondbeta}
\sum_{j\in\GZ} |c_j|^{\beta} < +\infty,
\end{equation}
then
\[ S_n(t) \indees A\cdot Z(t).\]
\end{cor}
\begin{proof}
We have to check (\ref{ajcondfddsimp}). By simple manipulations and taking into account that
due to (\ref{aje6})
$K =\sup_n n a_n^{-\alpha} h(a_n) < +\infty$ we obtain
\[\begin{split}
& \sum_{j=-\infty}^{-j_n} \frac{\big|d_{n,j}\big|^{\alpha}}{a_n^{\alpha}}
h\Big( \frac{a_n}{\big|d_{n,j}\big|}\Big)  \\
& \hspace{15mm}=
\frac{1}{n} \sum_{j=-\infty}^{-j_n}\big|\sum_{k=1-j}^{n-j}c_k\big|^{\beta}
\frac{n h(a_n)}{a_n^{\alpha}} \big|d_{n,j}\big|^{\alpha-\beta}
\frac{1}{h(a_n)} h\Big( \frac{a_n}{\big|d_{n,j}\big|}\Big)\\
&\hspace{15mm} \leq K  \frac{1}{n} \sum_{j=-\infty}^{-j_n} \sum_{k=1-j}^{n-j}|c_k|^{\beta}
\Psi_{\alpha - \beta}\Big(a_n, \frac{a_n}{\big|d_{n,j}\big|}\Big),
\end{split}
\]
where
\[ \Psi_{\alpha - \beta}(x,y) = \Big(\frac{x}{y}\Big)^{\alpha -\beta} \frac{h(y)}{h(x)}.\]
Let
\[ h(x) = c(x) \exp\big( \int_a^x \frac{\epsilon(u)}{u} \,du\big),\]
where $\lim_{x\to \infty} c(x) = c\in (0,\infty)$ and $\lim_{x\to\infty} \epsilon(x) = 0$, be the Karamata representation of the slowly varying function $h(x)$ (see e.g.
Theorem 1.3.1 in \cite{BGT87}). Take  $0 < \gamma < \min\{\alpha - \beta, c\}$ and let $L >a$ be such that for $x > L$
\[ \epsilon (x) \leq \gamma \text{  and  } c -\gamma < c(x) < c + \gamma.\]
Then we have for $x \geq y \geq L$
\[\frac{h(y)}{h(x)} = \frac{c(y)}{c(x)} \exp\big( \int_y^x \frac{\epsilon(u)}{u}\,du\big) \leq \frac{c + \gamma}{c - \gamma} \exp\Big(\gamma \log \Big( \frac{x}{y}\Big)\Big) = \frac{c + \gamma}{c - \gamma} \Big( \frac{x}{y}\Big)^{\gamma},\]
and so
\[ \Psi_{\alpha-\beta}(x,y) \leq K \Big(\frac{y}{x}\Big)^{\alpha - \beta-\gamma},\ x \geq y \geq L.\]
It follows from that fact and (\ref{ajT3e0}) that
\[\sup_{j\leq -j_n} \Psi_{\alpha - \beta}\Big(a_n, \frac{a_n}{\big|d_{n,j}\big|}\Big) \to 0, \text{ as $n\to\infty$}.\]
Hence it is sufficient to show that
\[ \sup_n\frac{1}{n} \sum_{j=-\infty}^{-j_n} \sum_{k=1-j}^{n-j}|c_k|^{\beta} < +\infty.\]
In fact, more is true.
\begin{lem}\label{L4} If $\sum_{j=0}^{\infty} |b_j| < +\infty$, then for each $t > 0$
\[\frac{1}{n} \sum_{j=0}^{\infty} \sum_{k=1+j}^{n+j}b_k \to 0, \text{ as $n\to\infty$.}\]
\end{lem}
\noindent {\em Proof of Lemma \ref{L4}} We have
\[\begin{split}
\Big|\frac{1}{n} \sum_{j=0}^{\infty} \sum_{k=1+j}^{n+j}b_k \Big|
&\leq \frac{1}{n} \sum_{j=0}^{\infty} \sum_{k=1+j}^{n+j}|b_k| \\
& = \frac{1}{n} \sum_{k=1}^{\infty} (k\wedge n) |b_k| \\
& =  \Big(\frac{1}{n} \sum_{k=1}^{n} k |b_k| + \sum_{k=n+1}^{\infty} |b_k|\Big).
\end{split}\]
The first sum in the last line converges to $0$ by Kronecker's lemma. The second is the
rest of a convergent series.

Returning to the proof of Corollary \ref{ThADR}, let us notice that convergence
\[\sum_{j=n+ j_n}^{+\infty} \frac{\big|d_{n,j}\big|^{\alpha}}{a_n^{\alpha}}
h\Big( \frac{a_n}{\big|d_{n,j}\big|}\Big) \to 0,\ \text{ as $n\to\infty$,}\]
can be checked the same way.
\end{proof}
\begin{cor}\label{Alphabiggerthanone} Under the usual conditions, if $\alpha \in (1,2)$ and $\sum_{j\in\GZ} |c_j| < +\infty$,
then
\[ S_n(t) \indees A\cdot Z(t).\]
\end{cor}

\begin{rem}\label{Rem4}
Corollaries \ref{ThADR} and \ref{Alphabiggerthanone} were proved independently
by Astrauskas \cite{Ast83} and Davis and Resnick \cite  {DaRe85}.
Our approach follows direct manipulations of Astrauskas,
while Davis and Resnick involved point process techniques.
\end{rem}

\begin{rem}\label{Rem5}
For $\alpha \leq 1$ assumption (\ref{ajcondbeta}) is unsatisfactory, for it excludes the case of strictly $\alpha$-stable random variables $\{Y_j\}$
with $\sum_j |c_j|^{\alpha} < +\infty$, but
$\sum_j |c_j|^{\beta} = +\infty$ for every $\beta < \alpha$.
With our criterion given in Theorem \ref{Th1} we can easily prove the needed result.
\end{rem}

\begin{cor}\label{Alphaleqone}
Suppose that $\alpha \leq 1$, $\sum_{j\in\GZ} |c_j|^{\alpha} < +\infty$, the usual conditions hold and $h$ is such that
\begin{equation}\label{hboundone}
h(\lambda x)/h(x) \leq M,\  \lambda \geq 1, x \geq x_0,
\end{equation}
for some constants $M$, $x_0$.
If the linear process $\{X_i\}$ is well-defined, then
\[ S_n(t) \indees A\cdot Z(t).\]
\end{cor}

\noindent{\em Proof of Corollary \ref{Alphaleqone}}
First notice that $\sum_j |c_j| < +\infty$ so that $A$ is defined.
Proceeding like in the proof of Corollary \ref{ThADR} we obtain
\[\begin{split}
& \sum_{j=-\infty}^{-j_n} \frac{\big|d_{n,j}\big|^{\alpha}}{a_n^{\alpha}}
h\Big( \frac{a_n}{\big|d_{n,j}\big|}\Big)  \\
& \hspace{15mm}=
 \frac{1}{n}\sum_{j=-\infty}^{-j_n} \Big|\sum_{k=1-j}^{n-j}c_k\Big|^{\alpha}
\frac{n h(a_n)}{a_n^{\alpha}}
\frac{1}{h(a_n)} h\Big( \frac{a_n}{\big|d_{n,j}\big|}\Big)\\
&\hspace{15mm} \leq K\cdot M \ \frac{1}{n} \sum_{j=-\infty}^{-j_n} \sum_{k=1-j}^{n-j}
|c_k|^{\alpha} \to 0,
\end{split}
\]
where the convergence to $0$ holds by Lemma \ref{L4}.

\begin{rem}\label{Rem55}
As mentioned before, the above corollary covers the important case when
$ h(x) \to C > 0$, as $x\to\infty$,
i.e. when the law of $Y_i$ is in the domain of {\em strict (or normal)} attraction. Many other examples
can be produced using Karamata's representation of slowly varying functions.
Assumption (\ref{hboundone}) is much in the spirit of Lemma A.4 in \cite{MS00}.
Our final result goes in different direction.
\end{rem}

\begin{rem}\label{Rem6}
Notice that if $\alpha <1$, then $\sum_j |c_j|^{\alpha} h(|c_j|^{-1}) < +\infty$, with $h$ slowly varying,  automatically
implies $\sum_j |c_j| < +\infty$.
\end{rem}

\begin{cor}\label{Corfin} Under the usual conditions, if
$\alpha < 1$, then
\[ S_n(t) \indees A\cdot Z(t),\]
if
\[ \sum_{j\in\GZ} |c_j|^{\alpha} < +\infty,\]
and the coefficients $c_j$ are regular in a very weak sense:
there exists a constant $0 < \gamma <\alpha$ such that
\begin{align}\label{ajemaxone}
 \frac{\max_{j+ 1 \leq k \leq j+n} |c_k|^{\frac{(1-\alpha)(\alpha -\gamma)}{(1 - \alpha + \gamma)}}}{\sum_{k = j+1}^{j+n} |c_k|^{\alpha}} &\leq K_+< +\infty, \ j \geq 0.\\
\frac{\max_{j-n \leq k\leq j-1} |c_k|^{\frac{(1-\alpha)(\alpha -\gamma)}{(1 - \alpha + \gamma)}}}{\sum_{k= j -n}^{j-1} |c_k|^{\alpha}} &\leq K_- < +\infty, \ j \leq 0.\label{ajemaxtwo}
\end{align}
(with the convention that $0/0 \equiv 1$.)
\end{cor}
\begin{rem}
Notice that we always assume that the linear process is well defined. This may require more than demanded in Corollary \ref{Corfin}.
\end{rem}
\noindent {\em Proof of Corollary \ref{Corfin}} As before, we have to check  (\ref{ajcondfddsimp}).

\[\begin{split}
& \sum_{j=-\infty}^{-j_n} \frac{\big|d_{n,j}\big|^{\alpha}}{a_n^{\alpha}}
h\Big( \frac{a_n}{\big|d_{n,j}\big|}\Big)  \\
& \hspace{15mm}=
\frac{1}{n} \sum_{j=-\infty}^{-j_n} \Big|\sum_{k=1-j}^{n-j}c_k\Big|^{\alpha - \gamma}
\frac{n h(a_n)}{a_n^{\alpha}} \big|d_{n,j}\big|^{\gamma}
\frac{1}{h(a_n)} h\Big( \frac{a_n}{\big|d_{n,j}\big|}\Big)\\
&\hspace{15mm} \leq K  \frac{1}{n} \sum_{j=-\infty}^{-j_n}  \Big|\sum_{k=1-j}^{n-j}c_k\Big|^{\alpha - \gamma} \Psi_{\gamma}\Big(a_n, \frac{a_n}{\big|d_{n,j}\big|}\Big),
\end{split}
\]
where $\Psi_{\gamma}(x,y)$ was defined in the proof of Corollary \ref{ThADR} and
\[\sup_{j\leq -j_n} \Psi_{\gamma}\Big(a_n, \frac{a_n}{\big|d_{n,j}\big|}\Big) \to 0, \text{ as $n\to\infty$}.\]
Thus it is enough to prove
\[
\sup_n \frac{1}{n} \sum_{j=-\infty}^{-j_n}  \Big|\sum_{k=1-j}^{n-j}c_k\Big|^{\alpha - \gamma} < +\infty.
\]
We have
\begin{equation}\label{ajekey}
\Big|\sum_{k=1-j}^{n-j}c_k\Big| \leq  \sum_{k=1-j}^{n-j}|c_k|
\leq \Big(\sum_{k=1-j}^{n-j}|c_k|^{\alpha} \Big) \cdot \max_{ 1 - j \leq k \leq n-j}
|c_k|^{1-\alpha},
\end{equation}
hence
\[ \begin{split}
\frac{1}{n} \sum_{j=-\infty}^{-j_n}  \Big|\sum_{k=1-j}^{n-j}c_k\Big|^{\alpha - \gamma} & =
\frac{1}{n} \sum_{j=-\infty}^{-j_n} \Big(\sum_{k=1-j}^{n-j}|c_k|^{\alpha}\Big)
\frac{ \Big|\sum_{k=1-j}^{n-j}c_k\Big|^{\alpha - \gamma}}{\Big(\sum_{k=1-j}^{n-j}|c_k|^{\alpha}\Big)}\\
&\leq \frac{1}{n} \sum_{j=-\infty}^{-j_n} \Big(\sum_{k=1-j}^{n-j}|c_k|^{\alpha}\Big)
\frac{ \max_{ 1 - j \leq k \leq n-j}
|c_k|^{(1-\alpha)(\alpha - \gamma)}}{\Big(\sum_{k=1-j}^{n-j}|c_k|^{\alpha}\Big)^{1 -\alpha + \gamma}} \\
& \leq \big(K_+\big)^{1-\alpha + \gamma} \frac{1}{n} \sum_{j=-\infty}^{-j_n} \Big(\sum_{k=1-j}^{n-j}|c_k|^{\alpha}\Big) \to 0.
\end{split}
\]
This is again more than needed. The proof of
\[\sum_{j=n+ j_n}^{+\infty} \frac{\big|d_{n,j}\big|^{\alpha}}{a_n^{\alpha}}
h\Big( \frac{a_n}{\big|d_{n,j}\big|}\Big) \to 0,\ \text{ as $n\to\infty$.}
\]
goes the same way.

\begin{exa}
If $\alpha < 1$,
\[ |c_j| = \frac{1}{|j|^{1/\alpha} \log^{(1+\varepsilon)/\alpha} |j|},\ |j| \geq 3,\]
and $\{X_i\}$ is well-defined,
then under the usual conditions
\[ S_n(t) \indees A\cdot Z(t).\]
\end{exa}
\label{Sec2}
\begin{rem}\label{Peligrad}
In our considerations we search for conditions giving functional convergence
of $\{S_n(t)\}$ with {\em the same normalization} as $\{Z_n(t)\}$ (by $\{a_n\}$).
It is possible to provide examples of linear processes, which are convergent in the sense of
finite dimensional distribution with different normalization. Moreover, it is likely that also in
the heavy-tailed case one can obtain a complete description of the convergence of linear processes, as it is done by Peligrad and Sang \cite{PeSa13} in the case of innovations belonging to the domain of attraction of a normal distribution. We conjecture that whenever the limit is a stable L\'evy motion our functional approach can be adapted to the more general setting.
\end{rem}

\section{Some complements}\label{seccomp}
\subsection{$S$-continuous functionals}
A phenomenon of self-cancelling oscillations, typical for the $S$ topology, was described in Example \ref{typical}. This example shows that {\em supremum cannot be} continuous in
the $S$ topology. In fact, supremum is lower semi-continuous with respect to $S$, as many other popular functionals - see \cite{J97-EJP}, Corollary 2.10.  On the other hand {\em addition is } sequentially continuous and this property was crucial in consideration given in Section \ref{Sec3}.

Here is another positive example of an $S$-continuous functional.

Let $\mu$ be an atomless measure on $[0,1]$ and let $h : \GR^1 \to \GR^1$
be a continuous function. Consider a smoothing operation $s_{\mu,h}$ on $\GD([0,1])$
given by the formula
\[ s_{\mu,h}(x)(t) = \int_0^t h(x(s))\,d\mu(s).\]
Then $s_{\mu,h}(x)(\cdot)$ is a continuous function on $[0,1]$ and a slight modification
of the proof of Proposition 2.15 in \cite{J97-EJP} shows that the mapping
\[ \Big(\GD([0,1]), S\Big) \ni x \mapsto s_{\mu,h}(x) \in\Big( \GC([0,1]), \|\cdot\|_{\infty}\Big)\]
is continuous. In particular, if we set $\mu = \ell$ (the Lebesgue measure), $h(0) = 0$, $h(x) \geq 0$,
and suppose that $x_n \ines 0$, then
\[ \int_0^1 h(x_n(s))\,ds  \to 0.\]
In the case of linear processes such functionals lead to the following result.

\begin{cor}\label{corappl}
Under the conditions of Corollaries \ref{ThADR},
\ref{Alphabiggerthanone},
\ref{Alphaleqone} or \ref{Corfin} we have for any $\beta > 0$
\[ \frac{1}{n a_n^{\beta}} \sum_{k=1}^n \Big| \sum_{i=1}^k
\big(\sum_j c_{i-j} Y_j\big) - A Y_i\Big|^{\beta} \inprob 0.\]
\end{cor}

\noindent {\em Proof of Corollary \ref{corappl}} The expression to be analyzed has the form
\[\int_0^1 H_{\beta}\big(S_n(t) - A\cdot Z_n(t)\big)\,dt,\]
where $H_{\beta}(x) = |x|^{\beta}$ and by
(\ref{ajeonedim})
\[ S_n(t) - A\cdot Z_n(t) \infdd 0.\]
We have checked in the course of the proof of Theorem \ref{ThMain}, that
$\{S_n\}$ is uniformly $S$-tight.  By (\ref{aje3}) $\{A\cdot Z_n\}$ is uniformly
$J_1$-tight, hence also $S$-tight. Similarly as in the proof of Proposition \ref{PropFin}
we deduce that $\{S_n - A\cdot Z_n\}$ is uniformly $S$-tight. Now an application of
Proposition \ref{Prop3} gives
\[ S_n - A\cdot Z_n \indist 0,\]
on the Skorokhod space $\GD([0,1])$ equipped with the $S$ topology.

\subsection{An example related to convergence in the $M_1$ topology}
\label{seccompl}
In Introduction we provided an example of a linear process ($c_0=1, c_1 = -1$) for which no
Skorokhod's convergence is possible. In this example $A=0$ and the limit is degenerate, what might suggest that another, more appropriate norming is applicable, under which the phenomenon disappears. Here we give an example with a {\em non-degenerate limit} showing that in the general case $M_1$-convergence need not hold.

\begin{exa}
Let $c_0 =\zeta   > -c_1 = \xi >0$. Then $X_j = \zeta Y_j - \xi Y_{j-1}$ and defining $Z_n(t)$ by (\ref{aje3}) we obtain for $t\in [k/n, (k+1)/n)$
\[ S_n(t) = \frac{1}{a_n} \sum_{j=1}^{k} X_j = \frac{1}{a_n}  \big(\zeta Y_{k} - \xi Y_0\big)
+ \big(\zeta - \xi) Z_n((k-1)/n).\]
Clearly,  the f.d.d. limit $\{\big(\zeta - \xi) Z(t)\}$ is non-degenerate. We will show that the sequence $\{S_n(t)\}$ is not uniformly $M_1$-tight and so cannot converge to  $\{  \big(\zeta - \xi) Z(t)\}$ in the $M_1$ topology.

For the sake of simplicity let us assume that $Y_j$'s are non-negative and
\[ \bP\big( Y_1 > x\big) = x^{-\alpha}, \ x \geq 1,\]
with $\alpha < 1$. Then we can choose $a_n = n^{1/\alpha}$. Consider sets
\[ G_n = \bigcup_{j=0}^{n-1} \big\{ Y_j > \varepsilon_n a_n, Y_{j+1} > \varepsilon_n a_n\big\}.  \]
where $\varepsilon_n = n^{-1/(3\alpha)}$.
Then
\[
 \bP\big( G_n\big) \leq (n+1) \bP\big(  Y_i > \varepsilon_n a_n\big)^2 = (n+1) \varepsilon_n^{-2\alpha} \big(n^{1/\alpha}\big)^{-2\alpha} \longrightarrow 0.
\]
Notice that
\begin{equation}\label{equwaga}
\text{on $G_n^c$ there are no two consecutive values of $Y_j$ exceeding $\varepsilon_n a_n$.}
\end{equation}
Let us define $Y_{n,j} = Y_j    \I \{ Y_j > \varepsilon_n a_n\}$ and set for $t\in [k/n, (k+1)/n)$
\[ \widetilde{S}_n(t) =  \frac{1}{a_n}  \big(\zeta Y_{n,k} - \xi Y_{n,0}\big)
+ \frac{\zeta - \xi}{a_n} \sum_{j=1}^{k-1} Y_{n,j}.\]
 We have by (\ref{ajfolkexp})
\begin{align*}
 \bE\big[\sup_{t\in [0,1]} \big| S_n(t) - \widetilde{S}_n(t)\big|\big] &\leq \frac{\zeta}{a_n}\sum_{j=0}^n
\bE\big[ Y_j \I\{ Y_j \leq \varepsilon_n a_n\}\big] \\
&\leq C_1 \zeta \frac{(n+1) (\varepsilon_n a_n)^{1 - \alpha}}{a_n} \to 0.
\end{align*}
It follows that $\{S_n(t)\}$ are uniformly $M_1$-tight if, and only if, $\{\widetilde{S}_n(t)\}$ are.
Let $w^{M_1}(x,\delta)$ be given by (\ref{eqomega}). Since $\bP\big( G_n^c\big) \to 1$ we have for any $\delta > 0$ and $\eta > 0$
\[ \limsup_n \bP\big( w^{M_1}(\widetilde{S}_n(\cdot),\delta) > \eta \big) =  \limsup_n \bP\big( \{w^{M_1}(\widetilde{S}_n(\cdot),\delta) > \eta\}\cap G_n^c \big).\]
And on $G_n^c$, by the property (\ref{equwaga}) and if $2/n <\delta $  we have
\[ \omega(\widetilde{S}_n(\cdot),\delta) \geq \frac{1}{a_n} \big(\zeta - \xi\big) \max_j Y_{n,j}.\]
If $\eta/ (\zeta - \xi) > \varepsilon_n$, then
\begin{align*}
\bP\big((1/a_n) &\max_j Y_{n,j} > \eta/ \big(\zeta - \xi\big)\big)  \\
& =
\bP\big((1/a_n) \max_j Y_j > \eta/ \big(\zeta - \xi\big)\big) \\
& \longrightarrow 1 -
\exp\big( - \big((\zeta - \xi)/\eta\big)^{\alpha}\big) = \theta> 0.
\end{align*}
Hence for each $\delta > 0$
\[ \liminf_n \bP\big( w^{M_1}(\widetilde{S}_n(\cdot),\delta) > \eta \big) \geq \theta > 0,\]
and the sequence $\{\widetilde{S}_n(t)\}$ cannot be uniformly $M_1$-tight.
\end{exa}

\subsection{Linear space of convergent linear processes}
We can explore the machinery of Section \ref{Sec2} to obtain a natural
\begin{prop}\label{PropSum}
We work under the assumptions of Theorem \ref{Th1}.
Denote by ${\cal C}_Y$ the set of sequences $\{c_i\}_{i\in \GZ}$ such that
if
\[ X_i=\sum_{j \in \GZ}c_j Y_{i-j}, \quad i \in \GZ,\]
then
\[ S_n(t)=\frac{1}{a_n} \sum_{i=1}^{[nt]}X_i \infdd A\cdot Z(t),\]
with $A = \sum_{i\in\GZ} c_i$.

Then ${\cal C}_Y$ is a linear subspace of $\GR^{\GZ}$.
\end{prop}

\noindent{\em Proof of Proposition \ref{PropSum}} Closeness of ${\cal C}_Y$ under multiplication
by a number is obvious. So let us assume that $\{c_i'\}$ and $\{c_i''\}$ are elements of ${\cal C}_Y$.
By Theorem \ref{Th1} we have to prove that
\begin{equation}\label{ajesumproof}
\begin{split}
 \sum_{j=-\infty}^0
\bP\Big( \big|\sum_{k=1-j}^{n-j} (c_k' + c_k'') \big| |Y_j| > a_n\Big)
 &\to 0,\ \text{ as $n\to\infty$.}\\
 \sum_{j=n+ 1}^{\infty}
\bP\Big( \big|\sum_{k=1-j}^{n-j}(c_k' + c_k'') \big| |Y_j| > a_n\Big) &\to 0,\ \text{ as $n\to\infty$.}
\end{split}
\end{equation}
But
\[\begin{split}
 \sum_{j=-\infty}^0
\bP\Big( \big|&\sum_{k=1-j}^{n-j} (c_k' + c_k'') \big| |Y_j| > a_n\Big) \\
&\leq  \sum_{j=-\infty}^0
\bP\Big( \big|\sum_{k=1-j}^{n-j} c_k'\big| |Y_j| + \big|\sum_{k=1-j}^{n-j} c_k'' \big| |Y_j| > a_n\Big) \\
& \leq \sum_{j=-\infty}^0 \bP\Big( \big|\sum_{k=1-j}^{n-j} c_k'\big| |Y_j| > a_n/2\Big) + \sum_{j=-\infty}^0 \bP\Big(\big|\sum_{k=1-j}^{n-j} c_k'' \big| |Y_j| > a_n/2\Big).
\end{split}
\]
Now both terms tend to $0$ by Remark \ref{RemConst}. Identical reasoning can be used in the proof of the ``dual" condition in (\ref{ajesumproof}).

\subsection{Dependent innovations}
In the main results of the paper we studied only {\em independent} innovations $\{Y_j\}$.
It is however clear that the functional $S$-convergence can be obtained under much weaker assumptions. In order to apply crucial Proposition \ref{PropFin} we need only that
\[ S_n(t)  \infdd A\cdot Z(t),\]
and that
\[ T_n^+ \indist A_+\cdot Z, \ \text{ and }\  T_n^- \indist A_-\cdot Z, \]
on the Skorokhod space $\GD([0,1])$ equipped with the  $M_1$ topology. For the latter relations
Theorem 1 of \cite{LoRi11} seems to be an ideal tool for associated sequences (see our Proposition \ref{PropLR}). A variety of potential other possible examples is given in
\cite{MTK10}.

\section{Appendix}

We provide two results of a technical character. The first one is well-known (\cite{Ast83}) and is stated here for completeness. Proposition \ref{Prop1} might be of independent interest.
\begin{prop}\label{Prop0}
Let $\{Y_j\}$ be an i.i.d.sequence satisfying (\ref{aje4}), (\ref{aje7}) and (\ref{aje8}) and let $\{c_j\}$ be a sequence of numbers. Then the series $ \sum_{j\in\GZ} c_{j} Y_j$ is well defined if, and only if,
\begin{equation}\label{aje2a0}
\sum_{j\in\GZ} |c_{j}|^{\alpha} h(|c_{j}|^{-1}) < +\infty.
\end{equation}
\end{prop}
\begin{prop}\label{Prop1}
Let $\{Y_j\}$ be an i.i.d.sequence satisfying (\ref{aje4}), (\ref{aje7}) and (\ref{aje8}).
Consider an array $\{c_{n,j}\,;\, n\in \GN, j\in\GZ\}$ of numbers such that for each $n\in \GN$
\begin{equation}\label{aje2a}
\sum_{j\in\GZ} |c_{n,j}|^{\alpha} h(|c_{n,j}|^{-1}) < +\infty.
\end{equation}
Set $V_n = \sum_{j\in\GZ} c_{n,j} Y_j,\ n\in\GN$. Then
\begin{equation}\label{ajvntozero}
V_n \inprob 0
\end{equation}
 if, and only if,
\begin{equation}\label{ajseriestozero}
\sum_{j\in\GZ} |c_{n,j}|^{\alpha} h(|c_{n,j}|^{-1}) \to 0,\ \text{ as $n \to \infty$ }.
\end{equation}
\end{prop}

In the proofs we shall need some estimates which seem to be a part of the probabilistic folklore.
\begin{lem} \label{lemfolk}
Assume that
\[\bP\big( |Y| > x\big) = x^{-\alpha} h(x),\]
 where $h(x)$ is slowly
varying at $x = \infty$.
\begin{description}
\item{\bf (i)} If  $\alpha \in (0,2)$, then there exists a constant $C_2$, depending on $\alpha$ and the law of $Y$ such that
\begin{equation}\label{ajfolkvar}
 \bE\big[Y^2 \GI\big(|Y| \leq x\big)\big] \leq
C_2 x^{2-\alpha} h(x),\ x > 0.
\end{equation}
\item{\bf (ii)} If $\alpha \in (0,1)$, then there exists a constant $C_1$, depending on $\alpha$ and the law of $Y$ such that
\begin{equation}\label{ajfolkexp}
 \bE \big[|Y| \GI\big(|Y| \leq x\big)\big] \leq  C_1 x^{1-\alpha} h(x), \ x >0.
\end{equation}
\item{\bf (iii)} If $\alpha \in (1,2)$, then there is $x_0 > 0$, depending on the law of $Y$,
such that
\begin{equation}\label{ajfolktailexp}
\bE \big[|Y| \GI \big(|Y| > x \big)\big] \leq   \bE \big[|Y|\, \GI\big( x \leq x_0\big)\big] +
\frac{2 \alpha}{\alpha - 1} x^{1-\alpha} h(x), \ x > 0.
\end{equation}
\end{description}
\end{lem}
\begin{proof}
Take $\beta > \alpha$. Applying the direct half of Karamata's Theorem (Th. 1.5.11 \cite{BGT87})
we obtain
\[ \bE \big[|Y|^{\beta} \GI \big(|Y| \leq x\big)\big] = \beta \int_0^x t^{\beta - 1} \bP\big(|Y| > t\big)\,dt - x^{\beta} \bP\big( |Y| > t\big) \sim \frac{\alpha}{\beta - \alpha} x^{\beta - \alpha} h(x).\]
Hence there exists $x_0$ such that
\[ \bE \big[|Y|^{\beta} \GI \big(|Y| \leq x\big) \big]\leq  \frac{2\alpha}{\beta - \alpha} x^{\beta - \alpha} h(x), \ x > x_0.\]
If $0 < x \leq x_0$, then
\[ \bE \big[|Y|^{\beta} \GI \big(|Y| \leq x\big)\big] \leq x^{\beta} = x^{\beta} \frac{x^{- \alpha} h(x)}{\bP\big(|Y| > x\big)} \leq \frac{1}{\bP\big(|Y| > x_0\big)} x^{\beta -\alpha} h(x).\]
Setting $C_{\beta} = \max\{ 1/\bP\big(|Y| > x_0\big), 2\alpha/(\beta-\alpha)\}$ one obtains
both (\ref{ajfolkvar}) and (\ref{ajfolkexp}).

To get (\ref{ajfolktailexp}) we proceed similarly. First, by Karamata's Theorem
\[ \bE\big[|Y| \GI \big(|Y| > x \big) \big]= \int_x^{\infty} \bP\big(|Y| > t\big)\,dt + x \bP\big(|Y| > x\big) \sim \frac{\alpha}{\alpha - 1} x^{1-\alpha} h(x),\]
Hence for some $x_0$ we have
\[\bE |Y| \GI \big(|Y| > x \big) \leq  \frac{2 \alpha}{\alpha - 1} x^{1-\alpha} h(x), \ x > x_0.
\]
Since $\alpha > 1$, we have  $\bE \big[|Y|\big] < +\infty$ and  (\ref{ajfolktailexp}) follows.
\end{proof}

\noindent {\em Proof of Proposition \ref{Prop0}} We begin with specifying the conditions of the Kolmogorov Three Series Theorem in terms of our linear sequences. We have
\begin{equation}\label{ajcondone}
\sum_{j\in\GZ} \bP \big(| c_{j} Y_j | > 1\big)  = \sum_{j\in \GZ} (\frac{1}{|c_{j}|})^{-\alpha} h(|c_{j}|^{-1})  = \sum_{j\in\GZ} |c_{j}|^{\alpha}  h(|c_{j}|^{-1}).
\end{equation}
Applying (\ref{ajfolkvar}) we obtain
\begin{equation}\label{ajcondtwo}
\begin{split}
\sum_{j\in\GZ} \bV\text{ar} \big((c_{j} Y_j) \GI\big(|c_{j} Y_j| \leq 1\big)\big)
&\leq \sum_{j\in\GZ} \bE \big[\big(c_{j} Y_j\big)^2 \GI\big(|c_{j} Y_j| \leq 1\big)\big] \\
&=  \sum_{j\in\GZ} |c_{j}|^2 \bE \big[Y_j^2 I(|Y_j| \leq 1/|c_{j}|)\big]  \\
&\leq C_2  \sum_{j\in\GZ} |c_{j}|^2  (1/|c_{j}|)^{2-\alpha} h(|c_{j}|^{-1})  \\
&= C_2  \sum_{j\in\GZ} |c_{j}|^{\alpha} h(|c_{j}|^{-1}).
\end{split}
\end{equation}
Similarly, if $\alpha\in (0,1)$, then by  (\ref{ajfolkexp})
\begin{equation}\label{ajcondthreea}
\begin{split}
\sum_{j\in\GZ} \big| \bE \big[c_{j} Y_j \GI \big(|c_{j} Y_j| \leq 1\big)\big] \big| &\leq
\sum_{j\in\GZ} |c_{j}| \bE \big[|Y_j| \GI \big(|Y_j| \leq 1/|c_{j}|\big)\big] \\
&\leq C_1 \sum_{j\in\GZ} |c_{j}| (1/|c_{j}|)^{1-\alpha} h(|c_{j}|^{-1}) \\
& = C_1 \sum_{j\in\GZ} |c_{j}|^{\alpha} h(|c_{j}|^{-1}) .
\end{split}
\end{equation}
If $\alpha = 1$, then by the  symmetry we have $\bE \big[Y_j \GI\big(|Y_j| \leq a\big)\big] = 0$, $a > 0$, and the series of truncated  expectations trivially vanishes
\begin{equation}\label{ajcondthreeb}
\sum_{j\in\GZ} \bE \big[c_{j} Y_j \GI\big(|c_{j} Y_j| \leq 1\big)\big] = 0.
\end{equation}
For $\alpha \in (1,2)$  we have $\bE \big[X_j\big] = 0$ and by (\ref{ajfolktailexp})
\begin{equation}\label{ajcondthreec}
\begin{split}
\sum_{j\in\GZ} | \bE \big[c_{j} Y_j \GI \big(|c_{j} Y_j| \leq 1\big)\big] | &= \sum_{j\in\GZ}
 | \bE \big[c_{j} Y_j \GI \big(|c_{j} Y_j| > 1\big)\big] |\\
 &\leq
\sum_{j\in\GZ} |c_{j}| \bE \big[|Y_j| \GI \big(|Y_j| > 1/|c_{j}|\big) \big]\\
&\leq \bE \big[|Y|\big] \max_{j\in\GZ} |c_{j}|  \# \{ j\,;\, |c_{j}| \geq 1/x_0\} \\
&\quad +
\frac{2 \alpha}{\alpha - 1} \sum_{j\in\GZ} |c_{j}|  (1/|c_{j}|)^{1-\alpha} h(|c_{j}|^{-1})
\end{split}
\end{equation}
By (\ref{ajcondone}) - (\ref{ajcondthreec}) we obtain that $\sum_{j\in \GZ} |c_{j}|^{\alpha} h(|c_{j}|^{-1}) < +\infty$ if, and only if, all the assumptions of the Three Series Theorem are satisfied. Hence
$ \sum_{j\in\GZ} c_{j} Y_j$ is a.s. convergent if, and only if, (\ref{aje2a0}) holds.\\[2mm]

\noindent {\em Proof of Proposition \ref{Prop1}} By Proposition \ref{Prop0} all random variables
 $V_n = \sum_{j\in\GZ} c_{n,j} Y_j$ are well-defined.
Let us consider a decomposition of each $V_n$ into a sum of another three (convergent!) series:
\[\begin{split}
 V_n = & \sum_{j\in\GZ} \Big(c_{n,j} Y_j I(|c_{n,j} Y_j| \leq 1) - \bE \big[c_{n,j} Y_j I(|c_{n,j} Y_j| \leq 1)\big]\Big) \\
& + \sum_{j\in\GZ} \bE \big[c_{n,j} Y_j I(|c_{n,j} Y_j| \leq 1)\big] \\
& + \sum_{j\in\GZ} c_{n,j} Y_j I(|c_{n,j} Y_j| > 1) \\
= & V_{n,1} + V_{n,2} + V_{n,3}.
\end{split}\]
By (\ref{ajcondtwo}) we have
\[  \bV\text{ar} \big( V_{n,1}\big) \leq C_2  \sum_{j\in \GZ} |c_{n,j}|^{\alpha} h(|c_{n,j}|^{-1}) \to 0,\ \text{ as } n\to\infty, \]
if (\ref{ajseriestozero}) holds. Similarly $V_{n,2}\to 0$ by (\ref{ajcondthreea}) -  (\ref{ajcondthreec}). Finally, we have
\[
\begin{split}
\bP\big( V_{n,3} \neq 0) & \leq \bP\big(\bigcup_{j\in\GZ} \{|c_{n,j}Y_j| > 1\}\big) \\
& \leq \sum_{j\in\GZ} \bP\big(|c_{n,j}Y_j| > 1\big) \\
& = \sum_{j\in\GZ} |c_{n,j}|^{\alpha}  h(|c_{n,j}|^{-1}) \to 0 \ \text{ as $n \to \infty$.}
\end{split}
\]
We have proved the sufficiency part of Proposition \ref{Prop1}.

To prove the ``only if'' part, we show first that $V_n \inprob 0$ implies uniform infinitesimality of the coefficients, that is
\begin{equation}
\label{ajt2e2}
\sup_{j\in\GZ} |c_{n,j}| \to 0, \text{ as $n\to\infty$. }
\end{equation}
Let $\{\bar{Y}_j\}$ be an independent copy of $\{Y_j\}$.
If $\bar{V}_n = \sum_{j\in\GZ} c_{n,j}\bar{Y}_j$, then also $V_n -\bar{V}_n\inprob 0$ and these are series of {\em symmetric} random variables. For each $n$ select some
arbitrary $j_n\in\GZ$ and consider decomposition into {\em independent symmetric} random variables
\[V_n - \bar{V}_n = c_{n,j_n}(Y_{j_n} - \bar{Y}_{j_n}) +  \sum_{j\in\GZ, j\neq j_n} c_{n,j}(Y_j - \bar{Y}_j) = W_n + \widetilde{W}_n.\]
Since $\{V_n - \bar{V}_n\}_{n\in\GN}$ is uniformly tight, so is   $\{W_n\}_{n\in\GN}$ (it follows from the L\'evy-Ottaviani inequality, see e.g. Proposition 1.1.1 in \cite{KwWo92}).
Since the law of $Y_j - \bar{Y}_j$ is non-degenerate we obtain
\[ \sup_n |c_{n,j_n}| < +\infty.\]
If along some subsequence $n'$ we would have $c_{n',j_{n'}} \to c \neq 0$, then for some $\theta \in \GR^1$
\[ \lim_{n'\to\infty} \bE \big[e^{i\theta W_{n'}}\big] = |\bE \big[e^{i\theta c Y}\big]|^2 < 1.\]
It follows that also
\[ \lim_{n'\to\infty} \bE \big[e^{i\theta ( V_{n'} - \bar{V}_{n'})}\big] = \lim_{n'\to\infty} \bE \big[e^{i\theta W_{n'}}\big] \bE \big[e^{i\theta \widetilde{W}_{n'}}\big] < 1.\]
This is in contradiction with $ V_n - \bar{V}_n\inprob 0$. Hence $c = 0$, $c_{n,j_n} \to 0$ and since $j_n$ was chosen arbitrary, (\ref{ajt2e2}) follows.

Now let us choose $k_n$ such that both
\[ \sum_{|j| > k_n} c_{n,j} Y_j \inprob 0,\ \text{ as } n\to\infty,\]
and
\[ \sum_{|j| > k_n} \bP\big( |c_{n,j} Y_j| > 1\big) \to 0,  \ \text{ as } n\to\infty.\]
Then $\{ X_{n,j} = c_{n,j} Y_j\,;\, |j| \leq k_n, n\in\GN\}$ is an {\em infinitesimal}
array of row-wise independent random variables, with row sums convergent in probability to zero. Applying the general central limit theorem (see e.g. Theorem 5.15 in \cite{Kall02}) we obtain
\[ \sum_{|j|\leq k_n} \bP \big( |X_{n,j}| > 1\big) = \sum_{|j|\leq k_n} \bP \big( |c_{n,j} Y_j| > 1\big) = \sum_{|j|\leq k_n}  |c_{n,j}|^{\alpha}
 h(|c_{n,j}|^{-1}) \to 0.\]
This completes the proof of Proposition \ref{Prop1}.

\vspace{3mm}
\noindent{\em Acknowledgement.}
The authors would like to thank the anonymous referee for careful reading of the manuscript and comments which improved the paper in various aspects.

\noindent{\sc Authors' addresses:}\\[2mm]
Raluca Balan \\
 Department of Mathematics and Statistics\\ University of Ottawa,\\
585 King Edward Avenue, Ottawa, ON, K1N 6N5, Canada.\\
rbalan@uottawa.ca\\[2mm]
Adam Jakubowski\\
Faculty of Mathematics and Computer Science\\
Nicolaus Copernicus University, \\
Chopina 12/18, 87-100 Torun, Poland.\\
adjakubo@mat.umk.pl \\[2mm]
Sana Louhichi\\
Laboratoire Jean Kuntzmann\\
Institut de math\'ematiques appliqu\'ees de Grenoble,\\
51 rue des Mathématiques,
F-38041 Grenoble, Cedex 9, France.\\
sana.louhichi@imag.fr

\end{document}